\documentclass[12pt]{amsart}
\usepackage{amsmath}
\usepackage{amssymb}
\usepackage{amsfonts}
\usepackage{amscd}
\usepackage{thmdefs}
\usepackage[pagebackref,hypertexnames=false, colorlinks, citecolor=red, linkcolor=red]{hyperref}
\usepackage[backrefs]{amsrefs}
\usepackage[color=green!15,bordercolor=red,linecolor=red!30]{todonotes}
\usepackage{ulem}

\input{tcilatex}
\theoremstyle{definition}
\theoremstyle{remark}
\numberwithin{equation}{section}

\input tcilatex

\begin{document}
\title{Composition of Dyadic Paraproducts}
\author[S. Pott]{Sandra Pott}
\thanks{}
\address{S. Pott, Centre for Mathematical Sciences, University of Lund,
Lund, Sweden}
\email{sandra@maths.lth.se}
\author[M.C. Reguera]{Maria Carmen Reguera$^1$}
\thanks{1.\,\,\,Research supported by grants 2009SGR-000420 (Generalitat de
Catalunya) and MTM-2010-16232 (Spain)}
\address{M. C. Reguera, Department of Mathematics, Universitat Aut\`onoma de
Barcelona, Barcelona, Spain}
\email{mreguera@mat.uab.cat}
\author[E. T. Sawyer]{Eric T. Sawyer$^2$}
\thanks{2.  Research supported in part by a NSERC Grant.}
\address{E. T. Sawyer, Department of Mathematics, McMaster University,
Hamilton, Canada}
\email{sawyer@mcmaster.ca}
\author[B. D. Wick]{Brett D. Wick$^3$}
\address{Brett D. Wick, School of Mathematics\\
Georgia Institute of Technology\\
686 Cherry Street\\
Atlanta, GA USA 30332-0160}
\email{wick@math.gatech.edu}
\urladdr{www.math.gatech.edu/~wick}
\thanks{3.  Research supported in part by National Science Foundation
DMS grants \# 1001098 and \# 955432.}
\thanks{The authors would like to thank the Banff International Research
Station for the Banff--PIMS Research in Teams support for the project: The
Sarason Conjecture and the Composition of Paraproducts.}

\begin{abstract}
We obtain necessary and sufficient conditions to characterize the
boundedness of the composition of dyadic paraproduct operators.
\end{abstract}

\maketitle
\tableofcontents

\section{Introduction}

Recall that a Toeplitz operator on the Hardy space of analytic functions $%
H^{2}(\mathbb{D})$ is defined by 
\begin{equation*}
T_{\varphi }:H^{2}(\mathbb{D})\rightarrow H^{2}(\mathbb{D})\,\text{\ where }%
\,T_{\varphi }f=\mathbb{P}_{H^{2}}\left( \varphi f\right) .
\end{equation*}%
It is well known that this operator is bounded if and only if $\varphi \in
L^{\infty }(\mathbb{T})$. Equivalently, the Toeplitz operator $T_{\varphi }$
is bounded if and only if $\sup_{\lambda \in \mathbb{D}}\left\Vert
T_{\varphi }k_{\lambda }\right\Vert _{H^{2}}<\infty $ where $k_{\lambda }(z)=%
\frac{1}{1-\overline{\lambda }z}$ is the reproducing kernel for $H^{2}(%
\mathbb{D})$. An infamous conjecture of Sarason, \cite{sarasonConj}, states
that the composition of two (potentially unbounded) Toeplitz operators is
bounded, i.e. $T_{\varphi }T_{\overline{\psi }}$ is a bounded operator, if
and only if a certain relatively simple testing condition on the symbols $%
\varphi $ and $\psi $ hold, see \cite{TVZ}. However, even though this
conjecture seems quite reasonable, a beautiful counterexample was
constructed by F. Nazarov in \cite{Naz} disproving this simple testing
condition.

In this paper we are interested in a discrete dyadic analogue of the Sarason
conjecture. This discrete problem is already very challenging and captures
much of the difficulty associated with Sarason's original conjecture but is
more amenable to study because of the dyadic nature of the problem. In
particular, we are concerned with dyadic Haar paraproducts, and obtaining
necessary and sufficient conditions for the boundedness of the composition
of two such paraproducts. The conditions characterizing the boundedness will
be much more general than just those characterizing boundedness for each
individual paraproduct - just as the condition $\left\Vert bd\right\Vert
_{\infty }<\infty $ that characterizes boundedness of the composition $%
M_{b}\circ M_{d}$ of pointwise multipliers is much more general than the
conditions $\left\Vert b\right\Vert _{\infty }<\infty $ and $\left\Vert
d\right\Vert _{\infty }<\infty $ that characterize individual boundedness of
the pointwise multipliers.

Let $\mathcal{D}$ denote the usual dyadic grid of intervals on the real
line. We consider sequences $b=\left\{ b_{I}\right\} _{I\in \mathcal{D}}$ of
complex numbers on $\mathcal{D}$, which we often refer to as \emph{symbols}.
Define the Haar function $h_{I}^{0}$ and averaging function $h_{I}^{1}$ by 
\begin{equation*}
h_{I}^{0}\equiv h_I\equiv\frac{1}{\sqrt{\left\vert I\right\vert }}\left( -%
\mathbf{1}_{I_{-}}+\mathbf{1}_{I_{+}}\right) \text{ and }h_{I}^{1}\equiv%
\frac{1}{\left\vert I\right\vert }\mathbf{1}_{I}\ ,\ \ \ \ \ I\in \mathcal{D}%
.
\end{equation*}%
The operators considered in this paper are the following dyadic paraproducts.

\begin{definition}
\label{Paraproduct_Def} Given a symbol $b=\left\{ b_{I}\right\} _{I\in 
\mathcal{D}}$ and a pair $\left( \alpha ,\beta \right) \in \left\{
0,1\right\} \times \left\{ 0,1\right\} $, define the \emph{dyadic paraproduct%
} acting on a function $f$ by 
\begin{equation*}
\mathsf{P}_{b}^{\left( \alpha ,\beta \right) }f\equiv \sum_{I\in \mathcal{D}%
}b_{I}\left\langle f,h_{I}^{\beta }\right\rangle_{L^2(\mathbb{R})}
h_{I}^{\alpha },
\end{equation*}%
where $h_{I}^{0}$ is the Haar function associated with $I$, and $h_{I}^{1}$
is the average function associated with $I$. The index $\left( \alpha ,\beta
\right) $ is referred to as the \emph{type} of $\mathsf{P}_{b}^{\left(
\alpha ,\beta \right) }$.
\end{definition}

The purpose of this paper is to characterize boundedness on $L^{2}\left( 
\mathbb{R}\right) $ of the compositions $\mathsf{P}_{b}^{\left( \alpha
,\beta \right) }\circ \mathsf{P}_{d}^{\left( \gamma ,\delta \right) }$. We
denote the composition $\mathsf{P}_{b}^{\left( \alpha ,\beta \right) }\circ 
\mathsf{P}_{d}^{\left( \gamma ,\delta \right) }$ by $\mathsf{P}%
_{b,d}^{\left( \alpha ,\beta ,\gamma ,\delta \right) }$, and refer to the
index $\left( \alpha ,\beta ,\gamma ,\delta \right) $ as the \emph{type} of
the product $\mathsf{P}_{b}^{\left( \alpha ,\beta \right) }\circ \mathsf{P}%
_{d}^{\left( \gamma ,\delta \right) }$. The dual $\left( \mathsf{P}%
_{b}^{\left( \alpha ,\beta \right) }\right) ^{\ast }=\mathsf{P}_{b}^{\left(
\beta ,\alpha \right) }$ of the operator $\mathsf{P}_{b}^{\left( \beta
,\alpha \right) }$ is obtained by exchanging exponents, which then reduces
the total number of products to be investigated. We are able to give
reasonable characterizations of the operator norm $\left\Vert \mathsf{P}%
_{b,d}^{\left( \alpha ,\beta ,\gamma ,\delta \right) }\right\Vert
_{L^{2}\left( \mathbb{R}\right) \rightarrow L^{2}\left( \mathbb{R}\right) }$
in two special cases, namely when the type of $\mathsf{P}_{b,d}^{\left(
\alpha ,\beta ,\gamma ,\delta \right) }$ is of the form $\left( \alpha
,0,0,\delta \right) $ or $\left( 0,\beta ,\gamma ,0\right) $.

In the first case, the product $\mathsf{P}_{b,d}^{\left( \alpha ,0,0,\delta
\right) }$ reduces to a single paraproduct $\mathsf{P}_{b\circ d}^{\left(
\alpha ,\delta \right) }$ whose symbol $b\circ d$ is built from the
sequences in a very simple manner.

In the second case, the compositions are not as easy since there is less
cancellation. However, we are able to transplant the problem, first to an
operator on the discrete Bergman space on $\mathcal{D}$, and then to a two
weight norm inequality for a positive or singular operator on $L^{2}\left( 
\mathcal{H}\right) $. The positive operator inequality reduces to the tree
inequality in \cite{ArRoSa}, while the singular operator inequality is solved by
an extension of a two weight theorem in \cite{NTV}. The transplantation idea
seems to be a novel element in our approach to paraproducts, and should find
application elsewhere.

Our main results are then the following theorems that characterize the
compositions in certain cases. To state them requires some additional
notation. For a sequence $a=\{a_I\}_{I\in\mathcal{D}}$ define. 
\begin{eqnarray*}
\left\Vert a\right\Vert _{\ell ^{\infty }} &\equiv &\sup_{I\in \mathcal{D}%
}\left\vert a_{I}\right\vert ; \\
\left\Vert a\right\Vert _{CM} &\equiv &\sqrt{\sup_{I\in \mathcal{D}}\frac{1}{%
\left\vert I\right\vert }\sum_{J\subset I}\left\vert a_{J}\right\vert^{2}}\ .
\end{eqnarray*}
Given two sequences $b=\{b_I\}_{I\in \mathcal{D}}$ and $d=\{d_I\}_{I\in 
\mathcal{D}}$ let $b\circ d$ denote the \emph{Schur product} of the
sequences, i.e. 
\begin{equation*}
b\circ d\equiv \left\{ b_{I}d_{I}\right\} _{I\in \mathcal{D}}\ .
\end{equation*}

In the case of a composition of type $(0,0,0,1)$, $(1,0,0,0)$ or $(0,0,0,0)$
we have the following characterization.

\begin{theorem}
\label{Theorem_Simple} The composition $\mathsf{P}_{b}^{\left( 0,0\right)
}\circ \mathsf{P}_{d}^{\left( 0,1\right) }$ and $\mathsf{P}%
_{b}^{\left(1,0\right) }\circ \mathsf{P}_{d}^{\left( 0,0\right) }$ is
bounded on $L^{2}\left( \mathbb{R}\right) $ if and only if $\left\Vert
b\circ d\right\Vert_{CM}<\infty$. Moreover, the operator norm of the
composition satisfies 
\begin{equation*}
\left\Vert\mathsf{P}_{b}^{\left( 0,0\right) }\circ \mathsf{P}_{d}^{\left(
0,1\right) }\right\Vert _{L^{2}(\mathbb{R})\rightarrow L^{2}(\mathbb{R})} =
\left\Vert\mathsf{P}_{b}^{\left( 1,0\right) }\circ \mathsf{P}_{d}^{\left(
0,0\right) }\right\Vert _{L^{2}(\mathbb{R})\rightarrow L^{2}(\mathbb{R}%
)}\approx \left\Vert b\circ d\right\Vert_{CM}.
\end{equation*}

\noindent The composition $\mathsf{P}_{b}^{\left( 0,0\right) }\circ \mathsf{P%
}_{d}^{\left( 0,0\right) }$ is bounded on $L^{2}\left( \mathbb{R}\right) $
if and only if $\left\Vert b\circ d\right\Vert_{\ell^\infty}<\infty$
Moreover, the operator norm of the composition satisfies 
\begin{equation*}
\left\Vert\mathsf{P}_{b}^{\left( 0,0\right) }\circ \mathsf{P}_{d}^{\left(
0,0\right) }\right\Vert _{L^{2}(\mathbb{R})\rightarrow L^{2}(\mathbb{R}%
)}\approx \left\Vert b\circ d\right\Vert_{\ell^\infty}.
\end{equation*}
\end{theorem}

For compositions of type $(1,0,0,1)$ we also have a characterization, but
again require some additional notation. Given a symbol $a=\left\{
a_{I}\right\} _{I\in \mathcal{D}}$, we define the \textit{sweep}, $\widehat{S%
}\left( a\right) $, of $a$ by%
\begin{equation}  \label{Sweep}
\widehat{S}\left( a\right) \equiv \left\{ \left\langle \sum_{J\in \mathcal{D}%
}a_{J}h_{J}^{1},h_{I}\right\rangle_{L^2(\mathbb{R})} \right\} _{I\in 
\mathcal{D}}=\left\{ \sum_{J\subsetneqq I}a_{J}\widehat{h_{J}^{1}}\left(
I\right) \right\} _{I\in \mathcal{D}},
\end{equation}%
and also the sequence $E(a)$ by%
\begin{equation}  \label{Ea}
E(a)\equiv\left\{ \frac{1}{\left\vert J\right\vert }\sum_{I\subset
J}a_{I}\right\} _{J\in \mathcal{D}}.
\end{equation}
The characterization is then given by the following theorem.

\begin{theorem}
\label{CET-Type} The composition $\mathsf{P}_{b}^{\left( 1,0\right) }\circ 
\mathsf{P}_{d}^{\left( 0,1\right) }$ is bounded on $L^{2}\left( \mathbb{R}%
\right) $ if and only if $\left\Vert \widehat{S}(b\circ
d)\right\Vert_{CM}<\infty$ and $\left\Vert E(b\circ
d)\right\Vert_{\ell^\infty}<\infty$. Moreover, the operator norm of the
composition $\mathsf{P}_{b}^{\left( 1,0\right) }\circ \mathsf{P}_{d}^{\left(
0,1\right) }$ on $L^2\left(\mathbb{R}\right)$ satisfies 
\begin{equation*}
\left\Vert \mathsf{P}_{b}^{\left( 1,0\right) }\circ \mathsf{P}_{d}^{\left(
0,1\right) }\right\Vert _{L^{2}(\mathbb{R})\rightarrow L^{2}(\mathbb{R}%
)}=\left\Vert \mathsf{P}_{b\circ d}^{\left( 1,1\right) }\right\Vert _{L^{2}(%
\mathbb{R})\rightarrow L^{2}(\mathbb{R})}\approx \left\Vert \widehat{S}%
(b\circ d)\right\Vert_{CM}+\left\Vert E(b\circ d)\right\Vert _{\ell^\infty}.
\end{equation*}
\end{theorem}

In the case of the composition of type $(0,1,1,0)$ we obtain the following
theorem. To state the characterization again requires slightly more
notation. Given a function $f\in L^2(\mathbb{R})$ and an interval $I\in%
\mathcal{D}$ we let 
\begin{equation}  \label{Project_Fcn}
\mathsf{Q}_I f\equiv\sum_{J\subset I} \left\langle f,h_J\right\rangle_{L^2(%
\mathbb{R})} h_J
\end{equation}
denote the projection of the function $f$ onto the span of the Haar
functions supported within the interval $I$. When applied to a sequence $%
a=\{a_I\}_{I\in\mathcal{D}}$ the operator $\mathsf{Q}_I$ takes the following
form: 
\begin{equation}  \label{Project_Seq}
\mathsf{Q}_I a\equiv\sum_{J\subset I} a_J h_J.
\end{equation}
Notice that this definition encompasses the definition when applied to
functions since we can always identify a function with its sequence of Haar
coefficients.

Our characterization is then the following theorem.

\begin{theorem}
\label{Theorem_Positive} The composition $\mathsf{P}_{b}^{\left( 0,1\right)
}\circ \mathsf{P}_{d}^{\left( 1,0\right) }$ is bounded on $L^{2}\left( 
\mathbb{R}\right) $ if and only if both%
\begin{eqnarray*}
\left\Vert \mathsf{Q}_{I}\mathsf{P}_{b}^{\left( 0,1\right) }\mathsf{P}%
_{d}^{\left( 1,0\right) }\left( \mathsf{Q}_{I}\overline{d}\right)
\right\Vert _{L^{2}(\mathbb{R})}^{2} & \leq & C_1^2 \left\Vert \mathsf{Q}%
_{I}d\right\Vert _{L^{2}(\mathbb{R})}^{2}; \\
\left\Vert \mathsf{Q}_{I}\mathsf{P}_{d}^{\left( 0,1\right) }\mathsf{P}%
_{b}^{\left( 1,0\right) }\left( \mathsf{Q}_{I}\overline{b}\right)
\right\Vert _{L^{2}(\mathbb{R})}^{2} & \leq & C_2^2\left\Vert \mathsf{Q}%
_{I}b\right\Vert _{L^{2}(\mathbb{R})}^{2}
\end{eqnarray*}
for all $I\in\mathcal{D}$; i.e. for all $I\in\mathcal{D}$ the following
inequalities are true 
\begin{eqnarray}
\sum_{J\subset I}\left\vert b_{J}\right\vert^{2}\frac{1}{\left\vert
J\right\vert ^{2}}\left( \sum_{L\subset J}\left\vert
d_{L}\right\vert^{2}\right) ^{2} &\leq & C_1^2\sum_{L\subset I}\left\vert
d_{L}\right\vert^{2};  \notag \\
\sum_{J\subset I}\left\vert d_{J}\right\vert^{2}\frac{1}{\left\vert
J\right\vert ^{2}}\left( \sum_{L\subset J}\left\vert
b_{L}\right\vert^{2}\right) ^{2} & \leq & C_2^2 \sum_{L\subset I}\left\vert
b_{L}\right\vert^{2}.  \notag
\end{eqnarray}%
Moreover, the norm of $\mathsf{P}_{b}^{\left( 0,1\right) }\circ \mathsf{P}%
_{d}^{\left( 1,0\right) }$ on $L^2\left(\mathbb{R}\right)$ satisfies 
\begin{equation*}
\left\Vert \mathsf{P}_{b}^{\left( 0,1\right) }\circ \mathsf{P}_{d}^{\left(
1,0\right) }\right\Vert _{L^{2}\left( \mathbb{R}\right) \rightarrow
L^{2}\left( \mathbb{R}\right) }\approx C_1+C_2
\end{equation*}
where $C_1$ and $C_2$ are the best constants in appearing above.
\end{theorem}

In the case of composition of type $(0,1,0,0)$, and by duality and symmetry
the type $(0,0,1,0)$, we have the following characterization of the
composition of Haar paraproducts.

\begin{theorem}
\label{Theorem_SIO} The composition $\mathsf{P}_{b}^{\left( 0,1\right)
}\circ \mathsf{P}_{d}^{\left( 0,0\right) }$ is bounded on $L^{2}\left( 
\mathbb{R}\right) $ if and only if both%
\begin{eqnarray*}
\left\vert d_I\right\vert\left\Vert \mathsf{P}_{b}^{\left( 0,1\right) }
h_I\right\Vert_{L^2\left(\mathbb{R}\right)} & \leq & C_1; \\
\left\Vert \mathsf{Q}_I \mathsf{P}_{d}^{\left( 0,0\right) } \mathsf{P}%
_{b}^{\left( 1,0\right) }\mathsf{Q}_{I}\overline{b}\right\Vert _{L^{2}\left(%
\mathbb{R}\right)} & \leq & C_2 \left\Vert \mathsf{Q}_{I} b\right\Vert
_{L^{2}\left(\mathbb{R}\right)} \\
\end{eqnarray*}
for all $I\in\mathcal{D}$.; i.e. for all $I\in\mathcal{D}$ the following
inequalities are true 
\begin{eqnarray*}
\left\vert d_I\right\vert\left(\frac{1}{\left\vert I\right\vert}%
\sum_{L\subsetneq I} \left\vert b_L\right\vert^2\right)^{\frac{1}{2}} & \leq
& C_1; \\
\left(\sum_{J\subset I}\frac{\left\vert d_J\right\vert^2}{\left\vert
J\right\vert}\left(\sum_{K\subset J_+} \left\vert
b_K\right\vert^2-\sum_{K\subset J_-} \left\vert
b_K\right\vert^2\right)^2\right)^{\frac{1}{2}} & \leq & C_2
\left(\sum_{L\subset I}\left\vert b_{L}\right\vert^{2}\right)^{\frac{1}{2}}.
\end{eqnarray*}
Moreover, the norm of $\mathsf{P}_{b}^{\left( 0,1\right) }\circ \mathsf{P}%
_{d}^{\left( 1,0\right) }$ on $L^2\left(\mathbb{R}\right)$ satisfies 
\begin{equation*}
\left\Vert \mathsf{P}_{b}^{\left( 0,1\right) }\circ \mathsf{P}%
_{d}^{\left(0,0\right) }\right\Vert _{L^{2}\left( \mathbb{R}\right)
\rightarrow L^{2}\left( \mathbb{R}\right) }\approx C_1+C_2
\end{equation*}
where $C_1$ and $C_2$ are the best constants in appearing above.
\end{theorem}

The outline of the paper is as follows. In Section \ref{SimpleCharacters} we
carry out the proof of Theorems \ref{Theorem_Simple} and \ref{CET-Type}.
These essentially reduce to the characterizations of when a single
paraproduct is bounded. In Section \ref{DifficultCharacters} we give the
proof of Theorems \ref{Theorem_Positive} and \ref{Theorem_SIO}. These
characterizations are more difficult, but can be studied via techniques used
to obtain two-weight inequalities for positive and well-localized operators.

As an application of these results it is possible to provide a new proof of
the following:

\begin{theorem}[Petermichl, \protect\cite{Petermichl}]
\label{linearA_2} Let $w\in A_2$. Then 
\begin{equation*}
\left\| H\right\|_{L^2(w)\rightarrow L^2(w)}\lesssim\left[w \right]_{A_2}.
\end{equation*}
\end{theorem}

Above, the $A_{2}$ characteristic of the function $w$ is the quantity: 
\begin{equation*}
\left[ w\right] _{A_{2}}\equiv \sup_{I\in \mathcal{D}}\left\langle
w\right\rangle _{I}\left\langle w^{-1}\right\rangle _{I};
\end{equation*}%
while the Hilbert transform is defined by 
\begin{equation*}
H(f)(x)\equiv \int_{\mathbb{R}}\frac{f(y)}{y-x}dy,
\end{equation*}%
with the integral taken in the principle value sense. One first notes 
\begin{equation*}
H:L^{2}(w)\rightarrow L^{2}(w)\ \ \ \mathnormal{\Longleftrightarrow \ \ \ }%
M_{w^{\frac{1}{2}}}HM_{w^{-\frac{1}{2}}}:L^{2}\rightarrow L^{2}.
\end{equation*}%
The second reduction is to note that since the Hilbert transform can be
recovered by averaging the Haar shifts, and since we are only after an upper
bound, it will be sufficient to study the following dyadic model operator 
\begin{equation*}
M_{w^{\frac{1}{2}}}\mathcal{S}M_{w^{-\frac{1}{2}}}:L^{2}\rightarrow L^{2}
\end{equation*}%
where $\mathcal{S}$ is a shift operator defined on the Haar basis by $%
\mathcal{S}h_{I}\equiv h_{I_{-}}-h_{I_{+}}$. Because of linearity, it
suffices to consider just \textquotedblleft half\textquotedblright\ of the
shift operator $\mathcal{S}$ defined by the operator $h_{I_{-}}\otimes h_{I}$%
. This averaging of shifts to recover $H$ is the key observation made by
Petermichl in \cite{Petermichl2} and played a decisive role in her proof of
Theorem \ref{linearA_2}. Then note that 
\begin{equation*}
M_{w^{\pm \frac{1}{2}}}=\mathsf{P}_{\widehat{w^{\pm \frac{1}{2}}}}^{(0,1)}+%
\mathsf{P}_{\widehat{w^{\pm \frac{1}{2}}}}^{(1,0)}+\mathsf{P}_{\langle
w^{\pm \frac{1}{2}}\rangle }^{(0,0)}
\end{equation*}%
where, we can recognize the operators above as paraproducts by setting 
\begin{eqnarray*}
\widehat{b}(I) &=&\left\langle b,h_{I}^{0}\right\rangle _{L^{2}\left( 
\mathbb{R}\right) } \\
\left\langle b\right\rangle _{I} &=&\left\langle b,h_{I}^{1}\right\rangle
_{L^{2}\left( \mathbb{R}\right) },
\end{eqnarray*}%
for $I\in \mathcal{D}$. Then writing 
\begin{equation}
\left( \mathsf{P}_{\widehat{w^{\frac{1}{2}}}}^{(0,1)}+\mathsf{P}_{\widehat{%
w^{\frac{1}{2}}}}^{(1,0)}+\mathsf{P}_{\langle w^{\frac{1}{2}}\rangle
}^{(0,0)}\right) \mathcal{S}\left( \mathsf{P}_{\widehat{w^{-\frac{1}{2}}}%
}^{(0,1)}+\mathsf{P}_{\widehat{w^{-\frac{1}{2}}}}^{(1,0)}+\mathsf{P}%
_{\langle w^{-\frac{1}{2}}\rangle }^{(0,0)}\right)   \label{Prod_Expan}
\end{equation}%
we can recognize these terms as composition of paraproducts. One then can
apply the above theorems characterizing the composition of paraproducts, and
then verify that the testing conditions that appear can all be controlled by
a linear power of the $A_{2}$ characteristic. In particular this strengthens
the linear bound for $\mathcal{S}$ on $L^{2}\left( w\right) $ by obtaining
such a bound for each of the nine operators arising in the canonical
decomposition (\ref{Prod_Expan}). A simpler proof of this, but again using
the strategy outlined above, appears in \cite{PSRW}.

\section{Reduction to Single Paraproducts}

\label{SimpleCharacters} Given two sequences $b=\{b_I\}_{I\in \mathcal{D}}$
and $d=\{d_I\}_{I\in \mathcal{D}}$ let $b\circ d$ denote the Schur product
of the sequences, i.e. 
\begin{equation*}
b\circ d\equiv \left\{ b_{I}d_{I}\right\} _{I\in \mathcal{D}}\ .
\end{equation*}%
The composition $\mathsf{P}_{b}^{\left( \alpha ,0\right) }\circ \mathsf{P}%
_{d}^{\left( 0,\beta \right) }$ is given by%
\begin{eqnarray}
\left( \mathsf{P}_{b}^{\left( \alpha ,0\right) }\circ \mathsf{P}_{d}^{\left(
0,\beta \right) }\right) f &=&\mathsf{P}_{b}^{\left( \alpha ,0\right)
}\left( \mathsf{P}_{d}^{\left( 0,\beta \right) }f\right) =\sum_{I\in 
\mathcal{D}}b_{I}\left\langle \mathsf{P}_{d}^{\left( 0,\beta \right)
}f,h_{I}^{0}\right\rangle_{L^2(\mathbb{R})} h_{I}^{\alpha }
\label{composition} \\
&=&\sum_{I\in \mathcal{D}}b_{I}\left\langle \sum_{J\in \mathcal{D}%
}d_{J}\left\langle f,h_{J}^{\beta }\right\rangle_{L^2(\mathbb{R})}
h_{J}^{0},h_{I}^{0}\right\rangle_{L^2(\mathbb{R})} h_{I}^{\alpha }  \notag \\
&=&\sum_{I\in \mathcal{D}}b_{I}d_{I}\left\langle f,h_{I}^{\beta
}\right\rangle_{L^2(\mathbb{R})} h_{I}^{\alpha }  \notag \\
&=&\mathsf{P}_{b\circ d}^{\left( \alpha ,\beta \right) }f\ .  \notag
\end{eqnarray}%
Thus the boundedness of the product $\mathsf{P}_{b}^{\left( \alpha ,0\right)
}\circ \mathsf{P}_{d}^{\left( 0,\beta \right) }$ reduces to that of a single
paraproduct $\mathsf{P}_{b\circ d}^{\left( \alpha ,\beta \right) }$. There
are three cases in which a single paraproduct is easily characterized,
namely $\mathsf{P}_{a}^{\left( 0,0\right) }$, $\mathsf{P}_{a}^{\left(
0,1\right) }$\ and $\mathsf{P}_{a}^{\left( 1,0\right) }$.

\begin{lemma}
\label{ClassicalCharacterization} We have the characterizations%
\begin{eqnarray}
\left\Vert \mathsf{P}_{a}^{\left( 0,0\right) }\right\Vert _{L^{2}(\mathbb{R}%
)\rightarrow L^{2}(\mathbb{R})} &=&\left\Vert a\right\Vert _{\ell ^{\infty
}};  \label{simple} \\
\left\Vert \mathsf{P}_{a}^{\left( 0,1\right) }\right\Vert _{L^{2}(\mathbb{R}%
)\rightarrow L^{2}(\mathbb{R})} &=&\left\Vert \mathsf{P}_{a}^{\left(
1,0\right) }\right\Vert _{L^{2}(\mathbb{R})\rightarrow L^{2}(\mathbb{R}%
)}\approx \left\Vert a\right\Vert _{CM}\ .  \label{CM}
\end{eqnarray}
\end{lemma}

\begin{proof}
With the notation%
\begin{equation*}
\widehat{f}\left( K\right) =\left\langle f,h_{K}\right\rangle_{L^2(\mathbb{R}%
)},
\end{equation*}%
the identities%
\begin{eqnarray*}
\left\Vert \mathsf{P}_{a}^{\left( 0,0\right) }f\right\Vert _{L^{2}(\mathbb{R}%
)}^{2} &=&\sum_{I,I^{\prime }\in \mathcal{D}}a_{I}\overline{a_{I^{\prime }}}%
\left\langle f,h_{I}\right\rangle_{L^2(\mathbb{R})} \overline{\left\langle
f,h_{I^{\prime }}\right\rangle_{L^2(\mathbb{R})}} \left\langle
h_{I},h_{I^{\prime }}\right\rangle_{L^2(\mathbb{R})} \\
& = & \sum_{I\in \mathcal{D}}\left\vert a_{I}\right\vert^{2}\left\vert 
\widehat{f}\left( I\right) \right\vert ^{2}, \\
\left\Vert f\right\Vert _{L^{2}(\mathbb{R})}^{2} &=&\sum_{I\in \mathcal{D}%
}\left\vert \widehat{f}\left( I\right) \right\vert ^{2},
\end{eqnarray*}%
immediately gives \eqref{simple}. Then the Carleson Embedding Theorem gives%
\begin{eqnarray*}
\left\Vert \mathsf{P}_{a}^{\left( 0,1\right) }f\right\Vert _{L^{2}(\mathbb{R}%
)}^{2} &=&\sum_{I,I^{\prime }\in \mathcal{D}}a_{I}\overline{a_{I^{\prime }}}%
\left\langle f,h_{I}^{1}\right\rangle_{L^2(\mathbb{R})} \overline{%
\left\langle f,h_{I^{\prime}}^{1}\right\rangle_{L^2(\mathbb{R})}}
\left\langle h_{I},h_{I^{\prime }}\right\rangle_{L^2(\mathbb{R})} \\
&=&\sum_{I\in \mathcal{D}}\left\vert a_{I}\right\vert^{2}\left\vert
\left\langle f,h_{I}^{1}\right\rangle_{L^2(\mathbb{R})} \right\vert
^{2}\lesssim \left\{ \sup_{I\in\mathcal{D}}\frac{1}{\left\vert I\right\vert }%
\sum_{J\subset I}\left\vert a_{J}\right\vert^{2}\right\} \left\Vert
f\right\Vert _{L^{2}(\mathbb{R})}^{2}.
\end{eqnarray*}%
So we have that 
\begin{equation*}
\left\Vert \mathsf{P}_{a}^{\left( 0,1\right) }\right\Vert _{L^{2}(\mathbb{R}%
)\rightarrow L^{2}(\mathbb{R})}\lesssim \left\Vert a\right\Vert_{CM}.
\end{equation*}
To see that the other inequality holds, simply test on a Haar function.
Indeed, let $\hat{I}$ denote the parent of $I$, and then we have 
\begin{equation*}
\left\| \mathsf{P}_{a}^{\left( 0,1\right) }h_{\hat{I}}\right\|_{L^2(\mathbb{R%
})}^2\leq \left\| \mathsf{P}_{a}^{\left( 0,1\right) }\right\|_{L^2(\mathbb{R}%
)\rightarrow L^2(\mathbb{R})}^2.
\end{equation*}
However, a computation shows that 
\begin{eqnarray*}
\left\| \mathsf{P}_{a}^{\left( 0,1\right) }h_{\hat{I}}\right\|_{L^2(\mathbb{R%
})}^2 & = & \left\| \sum_{J\in\mathcal{D}} a_J \left\langle h_{\hat{I}},
h_J^{1}\right\rangle h_J\right\|_{L^2(\mathbb{R})}^2 \\
& = & \sum_{J\in\mathcal{D}} \left\vert a_J\right\vert^2
\left\vert\left\langle h_{\hat{I}}, h_{J}^1\right\rangle_{L^2(\mathbb{R}%
)}\right\vert^2 \\
& \gtrsim & \frac{1}{\left\vert I\right\vert} \sum_{J\subset I}\left\vert
a_J\right\vert^2,
\end{eqnarray*}
which proves \eqref{CM}.
\end{proof}

It is clear that Lemma \ref{ClassicalCharacterization} coupled with the
computations above prove Theorem \ref{Theorem_Simple}.

\subsection{The Pott-Smith Identity}

We now recall a useful identity obtained by Pott and Smith in \cite{PS}*{%
Proposition 2.3} related to the composition of certain types of
paraproducts. We will use the identity to obtain necessary and sufficient
conditions for $\mathsf{P}_{a}^{\left(1,1\right)}$ to be bounded on $L^2(%
\mathbb{R})$.

We remind the reader that for a symbol $a=\left\{ a_{I}\right\} _{I\in 
\mathcal{D}}$, in equation \eqref{Sweep} the sweep $\widehat{S}\left(
a\right) $ of $a$ was defined by%
\begin{equation*}
\widehat{S}\left( a\right) \equiv \left\{ \left\langle \sum_{J\in \mathcal{D}%
}a_{J}h_{J}^{1},h_{I}\right\rangle_{L^2(\mathbb{R})} \right\} _{I\in 
\mathcal{D}}=\left\{ \sum_{J\subsetneqq I}a_{J}\widehat{h_{J}^{1}}\left(
I\right) \right\} _{I\in \mathcal{D}},
\end{equation*}%
and equation \eqref{Ea} we defined the sequence $E(a)$ by%
\begin{equation*}
E(a)\equiv\left\{ \frac{1}{\left\vert J\right\vert }\sum_{I\subset
J}a_{I}\right\} _{J\in \mathcal{D}}.
\end{equation*}

We now decompose the paraproduct $\mathsf{P}_{a}^{\left( 1,1\right) }$ into
paraproducts with simpler types, each having at least one $0$ in the index.
The most natural idea is to expand the averaging functions in a Haar series: 
$h_{I}^{1}=\sum_{J\supsetneqq I}\widehat{h_{I}^{1}}\left( J\right) h_{J}$,
and then to split the resulting double sum over intervals into diagonal,
upper and lower parts. Carrying out this strategy we obtain:%
\begin{eqnarray*}
\mathsf{P}_{a}^{\left( 1,1\right) }f &=&\sum_{I\in \mathcal{D}%
}a_{I}\left\langle f,h_{I}^{1}\right\rangle_{L^2(\mathbb{R})} h_{I}^{1} \\
& = & \sum_{I\in \mathcal{D}}a_{I}\left\langle f,\left( \sum_{J\supsetneqq I}%
\widehat{h_{I}^{1}}\left( J\right) h_{J}\right) \right\rangle_{L^2(\mathbb{R}%
)} \left( \sum_{K\supsetneqq I}\widehat{h_{I}^{1}}\left( K\right)
h_{K}\right) \\
&=&\left\{ \sum_{J\subsetneqq K}+\sum_{K\subsetneqq J}+\sum_{J=K}\right\}
\sum_{I\subset J\cap K}a_{I}\widehat{h_{I}^{1}}\left( J\right) \widehat{f}%
\left( J\right) \widehat{h_{I}^{1}}\left( K\right) h_{K} \\
&\equiv &T^{\left( 1,0\right) }f+T^{\left( 0,1\right) }f+T^{\left(
0,0\right) }f\ .
\end{eqnarray*}%
Now we have%
\begin{eqnarray*}
T^{\left( 0,0\right) }f &=&\sum_{J\in \mathcal{D}}\sum_{I\subset J}a_{I}%
\widehat{h_{I}^{1}}\left( J\right) \widehat{f}\left( J\right) \widehat{%
h_{I}^{1}}\left( J\right) h_{J}=\sum_{J\in \mathcal{D}}\sum_{I\subset
J}a_{I} \widehat{h_{I}^{1}}\left( J\right) ^{2}\widehat{f}\left( J\right)
h_{J} \\
&=&\sum_{J\in \mathcal{D}}\left( \frac{1}{\left\vert J\right\vert }%
\sum_{I\subset J}a_{I}\right) \widehat{f}\left( J\right) h_{J}=\mathsf{P}%
_{E(a)}^{\left( 0,0\right) }f
\end{eqnarray*}%
where we have use the definition of $E(a)$ in \eqref{Ea}. We also have with%
\begin{equation*}
\left\{ \sum_{I\subset J}a_{I}\widehat{h_{I}^{1}}\left( J\right) \right\}
_{J\in \mathcal{D}}=\left\{ \left\langle \sum_{I\subsetneqq
J}a_{I}h_{I}^{1},h_{J}\right\rangle_{L^2(\mathbb{R})} \right\} _{J\in 
\mathcal{D}}=\left\{ \widehat{\sum_{I\in \mathcal{D}}a_{I}h_{I}^{1}}\left(
J\right) \right\} _{J\in \mathcal{D}}
\end{equation*}%
that%
\begin{eqnarray*}
T^{\left( 1,0\right) }f &=&\sum_{J\subsetneqq K}\sum_{I\subset J}a_{I}%
\widehat{h_{I}^{1}}\left( J\right) \widehat{f}\left( J\right) \widehat{%
h_{I}^{1}}\left( K\right) h_{K}=\sum_{J\in \mathcal{D}}\sum_{I\subset
J}\sum_{K\supsetneqq J}a_{I}\widehat{h_{I}^{1}}\left( J\right) \widehat{f}%
\left( J\right) \widehat{h_{I}^{1}}\left( K\right) h_{K} \\
&=&\sum_{J\in \mathcal{D}}\left( \sum_{I\subset J}a_{I}\widehat{h_{I}^{1}}%
\left( J\right) \right) \widehat{f}\left( J\right) \left( \sum_{K\supsetneqq
J}\widehat{h_{J}^{1}}\left( K\right) h_{K}\right) \\
&=&\sum_{J\in \mathcal{D}}\widehat{\sum_{I\subset J}a_{I}h_{I}^{1}}\left(
J\right) \ \widehat{f}\left( J\right) \ h_{J}^{1}=\mathsf{P}_{\widehat{S}%
\left( a\right) }^{\left( 1,0\right) }f.
\end{eqnarray*}%
Similarly,%
\begin{equation*}
T^{\left( 0,1\right) }f=\mathsf{P}_{\widehat{S}\left( a\right) }^{\left(
0,1\right) }f,
\end{equation*}%
and altogether we have the desired decomposition%
\begin{equation}
\mathsf{P}_{a}^{\left( 1,1\right) }=\mathsf{P}_{\widehat{S}\left( a\right)
}^{\left( 1,0\right) }+\mathsf{P}_{\widehat{S}\left( a\right) }^{\left(
0,1\right) }+\mathsf{P}_{E(a)}^{\left( 0,0\right) }\ .  \label{PS decomp}
\end{equation}

Thus we see that the single paraproduct $\mathsf{P}_{a}^{\left( 1,1\right) }$
not covered by \eqref{simple} reduces to types already characterized. Using
this we can then obtain the characterization of the paraproduct $\mathsf{P}%
_{a}^{(1,1)}$.

\begin{corollary}
\label{ClassicalCharacterization2} The operator norm $\left\Vert \mathsf{P}%
_{a}^{\left( 1,1\right) }\right\Vert _{L^{2}\left( \mathbb{R}\right)
\rightarrow L^{2}\left( \mathbb{R}\right) }$ of $\mathsf{P}_{a}^{\left(
1,1\right) }$ on $L^{2}\left( \mathbb{R}\right) $ satisfies%
\begin{equation}
\left\Vert \mathsf{P}_{a}^{\left( 1,1\right) }\right\Vert _{L^{2}\left( 
\mathbb{R}\right) \rightarrow L^{2}\left( \mathbb{R}\right) }\approx
\left\Vert \widehat{S}(a)\right\Vert_{CM}+\left\Vert E(a)\right\Vert
_{\ell^\infty}.
\end{equation}
\end{corollary}

\begin{proof}
From \eqref{PS decomp} by applying Lemma \ref{ClassicalCharacterization} we
have the following estimate 
\begin{equation}  \label{Upper_Ineq}
\left\Vert \mathsf{P}_{a}^{\left( 1,1\right) }\right\Vert _{L^{2}\left( 
\mathbb{R}\right) \rightarrow L^{2}\left( \mathbb{R}\right) }\lesssim
\left\Vert \widehat{S}(a)\right\Vert_{CM}+\left\Vert E(a)\right\Vert
_{\ell^\infty}.
\end{equation}

We now turn to showing that inequality (\ref{Upper_Ineq}) can be reversed.
Suppose that $\left\Vert \mathsf{P}_{a}^{\left( 1,1\right) }\right\Vert
_{L^{2}\left( \mathbb{R}\right) \rightarrow L^{2}\left( \mathbb{R}\right) }$
is finite. Then an easy computation shows that 
\begin{eqnarray*}
\left\langle \mathsf{P}_{a}^{\left( 1,1\right) }h_{I},h_{I}\right\rangle
_{L^{2}(\mathbb{R})} &=&\left\langle \mathsf{P}_{\widehat{S}\left( a\right)
}^{\left( 1,0\right) }+\mathsf{P}_{\widehat{S}\left( a\right) }^{\left(
0,1\right) }+\mathsf{P}_{E(a)}^{\left( 0,0\right) }h_{I},h_{I}\right\rangle
_{L^{2}(\mathbb{R})} \\
&=&\left\langle \mathsf{P}_{\widehat{S}\left( a\right) }^{\left( 1,0\right)
}h_{I},h_{I}\right\rangle _{L^{2}(\mathbb{R})}+\left\langle \mathsf{P}_{%
\widehat{S}\left( a\right) }^{\left( 0,1\right) }h_{I},h_{I}\right\rangle
_{L^{2}(\mathbb{R})}+\left\langle \mathsf{P}_{E(a)}^{\left( 0,0\right)
}h_{I},h_{I}\right\rangle _{L^{2}(\mathbb{R})} \\
&=&E(a)_{I}
\end{eqnarray*}%
since 
\begin{equation*}
\left\langle \mathsf{P}_{\widehat{S}\left( a\right) }^{\left( 1,0\right)
}h_{I},h_{I}\right\rangle _{L^{2}(\mathbb{R})}=\widehat{S}\left( a\right)
_{I}\left\langle h_{I}^{1},h_{I}\right\rangle _{L^{2}(\mathbb{R})}=0.
\end{equation*}%
A similar computation demonstrates that $\left\langle \mathsf{P}_{\widehat{S}%
\left( a\right) }^{\left( 0,1\right) }h_{I},h_{I}\right\rangle _{L^{2}(%
\mathbb{R})}=0$ as well. Thus, we have 
\begin{equation}
\left\Vert E(a)\right\Vert _{\ell ^{\infty }}=\sup_{I\in \mathcal{D}%
}\left\vert E(a)_{I}\right\vert \leq \sup_{I\in \mathcal{D}}\left\vert
\left\langle \mathsf{P}_{a}^{\left( 1,1\right) }h_{I},h_{I}\right\rangle
_{L^{2}(\mathbb{R})}\right\vert \leq \left\Vert \mathsf{P}_{a}^{\left(
1,1\right) }\right\Vert _{L^{2}\left( \mathbb{R}\right) \rightarrow
L^{2}\left( \mathbb{R}\right) }.  \label{Eabounded}
\end{equation}%
Again, let $\hat{I}$ denote the parent of the dyadic interval $I$. Now set 
\begin{equation*}
F_{I}\equiv \sum_{J\subset I}\overline{\widehat{S}(a)_{J}}h_{J}.
\end{equation*}%
Then simple straightforward computations demonstrate that 
\begin{eqnarray*}
\left\Vert F_{I}\right\Vert _{L^{2}(\mathbb{R})}^{2} &=&\sum_{J\subset
I}\left\vert \widehat{S}(a)_{J}\right\vert ^{2}, \\
\left\langle F_{I},h_{\hat{I}}\right\rangle _{L^{2}(\mathbb{R})} &=&0, \\
\left\langle F_{I},h_{\hat{I}}^{1}\right\rangle _{L^{2}(\mathbb{R})} &=&0.
\end{eqnarray*}%
First, observe that 
\begin{equation}
\left\vert \left\langle \mathsf{P}_{a}^{(1,1)}F_{I},h_{\hat{I}}\right\rangle
\right\vert _{L^{2}(\mathbb{R})}\leq \left\Vert \mathsf{P}_{a}^{\left(
1,1\right) }\right\Vert _{L^{2}\left( \mathbb{R}\right) \rightarrow
L^{2}\left( \mathbb{R}\right) }\left( \sum_{J\subset I}\left\vert \widehat{S}%
(a)_{J}\right\vert ^{2}\right) ^{\frac{1}{2}}.  \label{Upper}
\end{equation}%
Next, observe that the computations above involving $F_{I}$ give that 
\begin{eqnarray*}
\left\langle P_{E(a)}^{(0,0)}F_{I},h_{\hat{I}}\right\rangle _{L^{2}(\mathbb{R%
})} &=&\sum_{K\in \mathcal{D}}E(a)_{K}\left\langle F_{I},h_{K}\right\rangle
_{L^{2}(\mathbb{R})}\left\langle h_{K},h_{\hat{I}}\right\rangle _{L^{2}(%
\mathbb{R})} \\
&=&E(a)_{\hat{I}}\left\langle F_{I},h_{\hat{I}}\right\rangle _{L^{2}(\mathbb{%
R})}=0;
\end{eqnarray*}%
and 
\begin{eqnarray*}
\left\langle P_{\widehat{S}(a)}^{(0,1)}F_{I},h_{\hat{I}}\right\rangle
_{L^{2}(\mathbb{R})} &=&\sum_{K\in \mathcal{D}}\widehat{S}(a)(K)\left\langle
F_{I},h_{K}^{1}\right\rangle _{L^{2}(\mathbb{R})}\left\langle h_{K},h_{\hat{I%
}}\right\rangle _{L^{2}(\mathbb{R})} \\
&=&\widehat{S}(a)(\hat{I})\left\langle F_{I},h_{\hat{I}}^{1}\right\rangle
_{L^{2}(\mathbb{R})}=0.
\end{eqnarray*}%
Thus, using \eqref{PS decomp}, \eqref{Upper}, and the computations above we
have that 
\begin{equation}
\left\vert \left\langle \mathsf{P}_{\widehat{S}(a)}^{(1,0)}F_{I},h_{\hat{I}%
}\right\rangle _{L^{2}(\mathbb{R})}\right\vert =\left\vert \left\langle 
\mathsf{P}_{a}^{(1,1)}F_{I},h_{\hat{I}}\right\rangle \right\vert _{L^{2}(%
\mathbb{R})}\leq \left\Vert \mathsf{P}_{a}^{\left( 1,1\right) }\right\Vert
_{L^{2}\left( \mathbb{R}\right) \rightarrow L^{2}\left( \mathbb{R}\right)
}\left( \sum_{J\subset I}\left\vert \widehat{S}(a)_{J}\right\vert
^{2}\right) ^{\frac{1}{2}}.  \label{Redux}
\end{equation}%
Finally, we compute 
\begin{eqnarray}
\left\vert \left\langle \mathsf{P}_{\widehat{S}(a)}^{(1,0)}F_{I},h_{\hat{I}%
}\right\rangle _{L^{2}(\mathbb{R})}\right\vert  &=&\left\vert \sum_{K\in 
\mathcal{D}}\widehat{S}(a)_{K}\left\langle F_{I},h_{K}\right\rangle _{L^{2}(%
\mathbb{R})}\left\langle h_{K}^{1},h_{\hat{I}}\right\rangle _{L^{2}(\mathbb{R%
})}\right\vert   \notag \\
&=&\left\vert \sum_{K\subset I}\left\vert \widehat{S}(a)_{K}\right\vert
^{2}\left\langle h_{K}^{1},h_{\hat{I}}\right\rangle _{L^{2}(\mathbb{R}%
)}\right\vert   \notag \\
&=&\frac{1}{\sqrt{\left\vert \hat{I}\right\vert }}\sum_{K\subset
I}\left\vert \widehat{S}(a)_{K}\right\vert ^{2}.  \label{Redux2}
\end{eqnarray}%
Combining \eqref{Redux} and \eqref{Redux2} yields 
\begin{equation*}
\frac{1}{\sqrt{\left\vert I\right\vert }}\sum_{K\subset I}\left\vert 
\widehat{S}(a)_{K}\right\vert ^{2}\lesssim \left\Vert \mathsf{P}_{a}^{\left(
1,1\right) }\right\Vert _{L^{2}\left( \mathbb{R}\right) \rightarrow
L^{2}\left( \mathbb{R}\right) }\left( \sum_{J\subset I}\left\vert \widehat{S}%
(a)_{J}\right\vert ^{2}\right) ^{\frac{1}{2}},
\end{equation*}%
which gives 
\begin{equation*}
\left( \frac{1}{\left\vert I\right\vert }\sum_{J\subset I}\left\vert 
\widehat{S}(a)_{J}\right\vert ^{2}\right) ^{\frac{1}{2}}\lesssim \left\Vert 
\mathsf{P}_{a}^{\left( 1,1\right) }\right\Vert _{L^{2}\left( \mathbb{R}%
\right) \rightarrow L^{2}\left( \mathbb{R}\right) }
\end{equation*}%
and then taking the supremum over $I\in \mathcal{D}$ gives. 
\begin{equation}
\left\Vert \widehat{S}(a)\right\Vert _{CM}\lesssim \left\Vert \mathsf{P}%
_{a}^{\left( 1,1\right) }\right\Vert _{L^{2}\left( \mathbb{R}\right)
\rightarrow L^{2}\left( \mathbb{R}\right) }.  \label{SaCM}
\end{equation}%
Combining \eqref{Eabounded} and \eqref{SaCM} gives 
\begin{equation}
\left\Vert E(a)\right\Vert _{\ell ^{\infty }}+\left\Vert \widehat{S}%
(a)\right\Vert _{CM}\lesssim \left\Vert \mathsf{P}_{a}^{\left( 1,1\right)
}\right\Vert _{L^{2}\left( \mathbb{R}\right) \rightarrow L^{2}\left( \mathbb{%
R}\right) }.  \label{Lower_Ineq}
\end{equation}%
Then \eqref{Upper_Ineq} and \eqref{Lower_Ineq} prove the Corollary.
\end{proof}

When studying the composition $\mathsf{P}_{b}^{\left( 1,0\right)}\circ 
\mathsf{P}_{d}^{\left( 0,1\right) }$ identity \eqref{composition} along with
Corollary \ref{ClassicalCharacterization2} yields the following result.

\begin{corollary}
The operator norm of the composition $\mathsf{P}_{b}^{\left( 1,0\right)
}\circ \mathsf{P}_{d}^{\left( 0,1\right) }$ satisfies 
\begin{equation*}
\left\Vert \mathsf{P}_{b}^{\left( 1,0\right) }\circ \mathsf{P}_{d}^{\left(
0,1\right) }\right\Vert _{L^{2}(\mathbb{R})\rightarrow L^{2}(\mathbb{R}%
)}=\left\Vert \mathsf{P}_{b\circ d}^{\left( 1,1\right) }\right\Vert _{L^{2}(%
\mathbb{R})\rightarrow L^{2}(\mathbb{R})}\approx \left\Vert \widehat{S}%
(b\circ d)\right\Vert_{CM}+\left\Vert E(b\circ d)\right\Vert _{\ell^\infty}.
\end{equation*}
\end{corollary}

It is clear that the above Corollary proves Theorem \ref{CET-Type}.

When the sequence $a=\{a_I\}_{I\in\mathcal{D}}$ is given by non-negative
terms, then we have the following estimate that will be useful as well. It
is proved simply by applying the Carleson Embedding Theorem.

\begin{proposition}
\label{CET_Prop} Let $a=\{a_I\}_{I\in\mathcal{D}}$ be a sequence of
non-negative numbers. Then 
\begin{equation}  \label{CET}
\left\Vert \mathsf{P}_{a}^{\left( 1,1\right) }\right\Vert _{L^{2}\left( 
\mathbb{R}\right) \rightarrow L^{2}\left( \mathbb{R}\right) }\lesssim
\left\|a^{\frac{1}{2}}\right\|_{CM}^2.
\end{equation}
\end{proposition}

\begin{proof}
Let $f,g\in L^2(\mathbb{R})$. Then we have 
\begin{eqnarray*}
\left\vert \left\langle \mathsf{P}_{a}^{(1,1)} f,g\right\rangle\right\vert &
= & \left\vert\sum_{I\in\mathcal{D}} a_I \left\langle
f,h_I^1\right\rangle_{L^2(\mathbb{R})} \left\langle
g,h_I^1\right\rangle_{L^2(\mathbb{R})}\right\vert \\
& \leq & \sum_{I\in\mathcal{D}} a_I \left\vert\left\langle
f,h_I^1\right\rangle_{L^2(\mathbb{R})} \left\langle
g,h_I^1\right\rangle_{L^2(\mathbb{R})}\right\vert \\
& \leq & \left(\sum_{I\in\mathcal{D}} a_I \left\vert\left\langle
f,h_I^1\right\rangle_{L^2(\mathbb{R})} \right\vert^2\right)^{\frac{1}{2}%
}\left(\sum_{I\in\mathcal{D}} a_I \left\vert\left\langle
g,h_I^1\right\rangle_{L^2(\mathbb{R})} \right\vert^2\right)^{\frac{1}{2}}.
\end{eqnarray*}
Now apply Lemma \ref{ClassicalCharacterization} and \eqref{CM} to see that 
\begin{equation*}
\left(\sum_{I\in\mathcal{D}} a_I \left\vert\left\langle
f,h_I^1\right\rangle_{L^2(\mathbb{R})} \right\vert^2\right)^{\frac{1}{2}%
}\lesssim \left\{ \sup_{I\in\mathcal{D}}\frac{1}{\left\vert I\right\vert }%
\sum_{J\subset I} a_{J}\right\}^{\frac{1}{2}}\left\Vert f\right\Vert_{L^2(%
\mathbb{R})},
\end{equation*}
and so we have 
\begin{equation*}
\left\vert \left\langle \mathsf{P}_{a}^{(1,1)} f,g\right\rangle\right\vert =
\left\vert\sum_{I\in\mathcal{D}} a_I \left\langle f,h_I^1\right\rangle_{L^2(%
\mathbb{R})} \left\langle g,h_I^1\right\rangle_{L^2(\mathbb{R}%
)}\right\vert\lesssim \left\{ \sup_{I\in\mathcal{D}}\frac{1}{\left\vert
I\right\vert }\sum_{J\subset I} a_{J}\right\}\left\Vert f\right\Vert_{L^2(%
\mathbb{R})}\left\Vert g\right\Vert_{L^2(\mathbb{R})}
\end{equation*}
However, 
\begin{equation*}
\left\|a^{\frac{1}{2}}\right\|_{CM}^2=\sup_{I\in\mathcal{D}}\frac{1}{%
\left\vert I\right\vert }\sum_{J\subset I} a_{J},
\end{equation*}
and so the Proposition follows.
\end{proof}

\section{Transplantation}

\label{DifficultCharacters}

We have%
\begin{equation*}
\mathsf{P}_{b}^{\left( \alpha ,\beta \right) }h_{I}=\sum_{K\in \mathcal{D}%
}b_{K}\left\langle h_{I},h_{K}^{\beta }\right\rangle_{L^2(\mathbb{R})}
h_{K}^{\alpha }=\left\{ 
\begin{array}{ccl}
b_{I}\,h_{I} & \text{ if } & \left( \alpha ,\beta \right) =\left( 0,0\right)
\\ 
b_{I}\,h_{I}^{1} & \text{ if } & \left( \alpha ,\beta \right) =\left(
1,0\right) \\ 
\sum_{K\varsubsetneqq I}b_{K}\,\widehat{h_{K}^{1}}\left( I\right)\, h_{K} & 
\text{ if } & \left( \alpha ,\beta \right) =\left( 0,1\right) \\ 
\sum_{K\varsubsetneqq I}b_{K}\,\widehat{h_{K}^{1}}\,\left( I\right) h_{K} & 
\text{ if } & \left( \alpha ,\beta \right) =\left( 1,1\right).%
\end{array}%
\right.
\end{equation*}%
From these formulas we can compute the Gram matrices of the composition of
paraproducts. We will then choose an appropriate representation of Hilbert
space on which to analyze a given Gram matrix. It is the simplicity of these
formulas when $\beta =0$ that accounts for our success in characterizing
boundedness of products with type $\left( 0,\beta ,\gamma ,0\right) $.

At this point we also set forth some notation that will be used through out
the remainder of this section.  For the dyadic grid $\mathcal{D}$ we let $\ell^2(\mathcal{D})$ denote the standard space of square integrable sequences indexed by the dyadic intervals.  For a weight function $\omega:\mathcal{D}\to\mathbb{R}_+$ we let $\ell^2(\omega)$ denote the sequences $\{a_I\}_{I\in\mathcal{D}}$ for which
$$
\sum_{I\in\mathcal{D}} \omega(I) \left\vert a_I\right\vert^2<\infty.
$$

Recall now that we can identify the dyadic
grid $\mathcal{D}$ on the real line with the standard Bergman tree of
Carleson tiles on the upper plane by associating each $I\in \mathcal{D}$
with the Carleson tile 
\begin{equation*}
T\left( I\right) \equiv I\times \left[ \frac{\left\vert I\right\vert }{2}%
,\left\vert I\right\vert \right].
\end{equation*}
Also set 
\begin{equation*}
Q\left( I\right) \equiv I\times \left[ 0,\left\vert I\right\vert \right]%
=\bigcup\limits_{J\subset I}T\left( J\right);
\end{equation*}
which is the Carleson square associated with $I\in \mathcal{D}$.

Let $\mathcal{H}$ denote the upper half plane, and so in particular we see
that $\mathcal{H}=\bigcup_{I\in\mathcal{D}} T\left (I \right)$. We will let $%
L^2(\mathcal{H})$ denote the standard $L^2$ space on the upper half plane,
and for a non-negative function $\sigma$ we will let $L^2(\mathcal{H}%
;\sigma) $ denote the functions that are square integrable with respect to $%
\sigma \,dA$, i.e, 
\begin{equation*}
\left\| f\right\|_{L^2\left(\mathcal{H}\right)}^2\equiv\int_{\mathcal{H}%
}\left\vert f(z)\right\vert^2 \,dA(z)\quad\mathnormal{ and } \quad \left\|
f\right\|_{L^2\left(\mathcal{H};\sigma\right)}^2\equiv\int_{\mathcal{H}%
}\left\vert f(z)\right\vert^2 \sigma(z)\,dA(z).
\end{equation*}

Now consider the Hilbert subspace $L^2_c\left(\mathcal{H}\right)$ which
denotes the set of functions that are square integrable on $\mathcal{H}$,
but are constant on tiles. Namely, $f:\mathcal{D}\rightarrow\mathbb{C}$ and
can be represented as 
\begin{equation*}
f=\sum_{I\in\mathcal{D}} f_I \mathbf{1}_{T(I)}.
\end{equation*}
Then we have that 
\begin{equation*}
L^2_c\left(\mathcal{H}\right)\equiv \left\{ f:\mathcal{D}\rightarrow \mathbb{%
C} :\sum_{I\in \mathcal{D}}\left\vert f\left( I\right) \right\vert
^{2}\left\vert I\right\vert ^{2}<\infty\right\} ,
\end{equation*}
with norm $\left\Vert f\right\Vert _{L^2_c\left(\mathcal{H}\right)}=\sqrt{%
\frac{1}{2}\sum_{I\in \mathcal{D}}\left\vert f\left( I\right) \right\vert
^{2}\left\vert I\right\vert ^{2}}$.

For $f\in L^2\left(\mathcal{H}\right)$, let $\widetilde{f}=\frac{f}{%
\left\Vert f\right\Vert _{L^2\left(\mathcal{H}\right)}}$ denote the
normalized function. Then it is immediate that $\left\{ \widetilde{\mathbf{1}%
}_{T\left( I\right) }\right\} _{I\in \mathcal{D}}$ is an orthonormal basis
of $L^2_c\left(\mathcal{H}\right)$ and easy to see that $\left\{ \widetilde{%
\mathbf{1}}_{Q\left( I\right) }\right\} _{I\in \mathcal{D}}$ is a Riesz
basis of $L^2_c\left(\mathcal{H}\right)$.

For $\lambda\in\mathbb{R}$ and $a\equiv\{a_I\}_{I\in\mathcal{D}}$ the
multiplication operator $\mathcal{M}_{a}^{\lambda }$ is defined on basis
elements $\widetilde{\mathbf{1}}_{T\left( K\right) }$ by%
\begin{equation*}
\mathcal{M}_{a}^{\lambda }\widetilde{\mathbf{1}}_{T\left(
K\right)}=a_{K}\left\vert K\right\vert ^{\lambda }\widetilde{\mathbf{1}}%
_{T\left( K\right) }.
\end{equation*}
Note that $\mathcal{M}_{a}^{-1}$ is \emph{not} the inverse of $\mathcal{M}%
_{a}$!. We will also let $\overline{b}\equiv\left\{\overline{b_{I}}%
\right\}_{I\in\mathcal{D}}$.

Recall that in \eqref{Project_Fcn} for an interval $I\in\mathcal{D}$ and a
function $f\in L^2(\mathbb{R})$ we let 
\begin{equation*}
\mathsf{Q}_I f\equiv\sum_{J\subset I} \left\langle f,h_J\right\rangle_{L^2(%
\mathbb{R})} h_J
\end{equation*}
denote the projection of the function $f$ onto the span of the Haar
functions supported within the interval $I$. For sequences $%
a\equiv\{a_I\}_{I\in\mathcal{D}}$ in \eqref{Project_Seq} the operator $%
\mathsf{Q}_I$ takes the following form: 
\begin{equation*}
\mathsf{Q}_I a\equiv\sum_{J\subset I} a_J h_J.
\end{equation*}


We now study each remaining composition type in turn.

\subsection{Type (0,\,1,\,1,\,0) Compositions}

\label{0110-type}

The Gram matrix $\mathfrak{G}_{\mathsf{P}_{b}^{\left( 0,1\right) }\circ 
\mathsf{P}_{d}^{\left( 1,0\right) }}=\left[ G_{I,J}\right] _{I,J\in \mathcal{%
D}}$ of the operator $\mathsf{P}_{b}^{\left( 0,1\right) }\circ \mathsf{P}%
_{d}^{\left( 1,0\right) }$ relative to the Haar basis $\left\{ h_{I}\right\}
_{I\in \mathcal{D}}$ has entries 
\begin{eqnarray*}
G_{I,J} &=&\left\langle \mathsf{P}_{b}^{\left( 0,1\right) }\circ \mathsf{P}%
_{d}^{\left( 1,0\right) }h_{J},h_{I}\right\rangle_{L^2(\mathbb{R})}
=\left\langle \mathsf{P}_{d}^{\left( 1,0\right) }h_{J},\mathsf{P}%
_{b}^{\left( 1,0\right) }h_{I}\right\rangle_{L^2(\mathbb{R})} \\
&=&\left\langle d_{J}\,h_{J}^{1},b_{I}\,h_{I}^{1}\right\rangle_{L^2(\mathbb{R%
})} \\
&=&\overline{b_{I}}d_{J}\frac{\left\vert I\cap J\right\vert }{\left\vert
I\right\vert \left\vert J\right\vert }=\left\{ 
\begin{array}{ccc}
\overline{b_{I}}d_{J}\frac{1}{\left\vert I\right\vert } & \text{ if } & 
J\subset I \\ 
\overline{b_{I}}d_{J}\frac{1}{\left\vert J\right\vert } & \text{ if } & 
I\subset J \\ 
0 & \text{ if } & I\cap J=\emptyset.%
\end{array}%
\right.
\end{eqnarray*}
Define an operator $\mathsf{T}_{b,d}^{\left( 0,1,1,0\right) }$ on $%
L^2_c\left(\mathcal{H}\right)$\ by%
\begin{equation*}
\mathsf{T}_{b,d}^{\left( 0,1,1,0\right) }\equiv \mathcal{M}_{\overline{b}%
}^{0}\left( \sum_{K\in \mathcal{D}}\widetilde{\mathbf{1}}_{T\left( K\right)
}\otimes \widetilde{\mathbf{1}}_{Q\left( K\right) }\right) \mathcal{M}%
_{d}^{-1}.
\end{equation*}%
Then the Gram matrix $\mathfrak{G}_{\mathsf{T}_{b,d}^{\left( 0,1,1,0\right)
}}=\left[ G_{I,J}\right] _{I,J\in \mathcal{D}}$ of $\mathsf{T}_{b,d}^{\left(
0,1,1,0\right) }$ relative to the basis $\left\{ \widetilde{\mathbf{1}}%
_{T\left( I\right) }\right\} _{I\in \mathcal{D}}$ has entries%
\begin{eqnarray*}
G_{I,J} &=&\left\langle \mathsf{T}_{b,d}^{\left( 0,1,1,0\right) }\widetilde{%
\mathbf{1}}_{T\left( J\right) },\widetilde{\mathbf{1}}_{T\left( I\right)
}\right\rangle_{L^2\left(\mathcal{H}\right)} \\
&=&\left\langle \mathcal{M}_{\overline{b}}\left( \sum_{K\in \mathcal{D}}%
\widetilde{\mathbf{1}}_{T\left( K\right) }\otimes \widetilde{\mathbf{1}}%
_{Q\left( K\right) }\right) \mathcal{M}_{d}^{-1}\widetilde{\mathbf{1}}%
_{T\left( J\right) },\widetilde{\mathbf{1}}_{T\left( I\right)
}\right\rangle_{L^2\left(\mathcal{H}\right)} \\
&=&\sum_{K\in \mathcal{D}}\left\langle \left\langle \widetilde{\mathbf{1}}%
_{Q\left( K\right) },\mathcal{M}_{d}^{-1}\widetilde{\mathbf{1}}_{T\left(
J\right) }\right\rangle_{L^2\left(\mathcal{H}\right)} \mathcal{M}_{\overline{%
b}}\widetilde{\mathbf{1}}_{T\left( K\right) }, \widetilde{\mathbf{1}}%
_{T\left( I\right) }\right\rangle_{L^2\left(\mathcal{H}\right)} \\
&=&\sum_{K\in \mathcal{D}}\overline{b_{K}}d_{J}\left\vert
J\right\vert^{-1}\left\langle \widetilde{\mathbf{1}}_{Q\left( K\right) },%
\widetilde{\mathbf{1}}_{T\left( J\right) }\right\rangle_{L^2\left(\mathcal{H}%
\right)} \left\langle \widetilde{\mathbf{1}}_{T\left( K\right) },\widetilde{%
\mathbf{1}}_{T\left( I\right) }\right\rangle_{L^2\left(\mathcal{H}\right)} \\
&=&\overline{b_{I}}d_{J}\left\vert J\right\vert^{-1}\left\langle \widetilde{%
\mathbf{1}}_{Q\left( I\right) },\widetilde{\mathbf{1}}_{T\left( J\right)
}\right\rangle_{L^2\left(\mathcal{H}\right)} \left\langle \widetilde{\mathbf{%
1}}_{T\left( I\right) },\widetilde{\mathbf{1}}_{T\left( I\right)
}\right\rangle_{L^2\left(\mathcal{H}\right)} \\
&=&\overline{b_{I}}d_{J}\sqrt{2}\frac{\left\vert Q\left( I\right) \cap
T\left( J\right) \right\vert }{\left\vert I\right\vert\left\vert
J\right\vert^2 }=\frac{1}{\sqrt{2}}\left\{ 
\begin{array}{ccc}
\overline{b_{I}}d_{J}\frac{1}{\left\vert I\right\vert } & \text{ if } & 
J\subset I \\ 
0 & \text{ if } & J\not\subset I.%
\end{array}%
\right.
\end{eqnarray*}%
Thus, up to an absolute constant, $\mathfrak{G}_{\mathsf{T}_{b,d}^{\left(
0,1,1,0\right) }}$ matches $\mathfrak{G}_{\mathsf{P}_{b}^{\left( 0,1\right)
}\circ \mathsf{P}_{d}^{\left( 1,0\right) }}$ in the lower triangle where $%
J\subset I$.

By the above computations we have 
\begin{equation*}
\left\Vert \mathsf{P}_{b}^{\left( 0,1\right) }\circ \mathsf{P}%
_{d}^{\left(1,0\right) }\right\Vert _{L^{2}\left( \mathbb{R}\right)
\rightarrow L^{2}\left( \mathbb{R}\right) }\leq \left\Vert \mathsf{T}%
_{b,d}^{\left(0,1,1,0\right) }\right\Vert _{L^2\left(\mathcal{H}%
\right)\rightarrow L^2\left(\mathcal{H}\right)}+\left\Vert \mathsf{T}%
_{d,b}^{\left( 0,1,1,0\right) }\right\Vert _{L^2\left(\mathcal{H}%
\right)\rightarrow L^2\left(\mathcal{H}\right)}
\end{equation*}
and we will further show below that 
\begin{eqnarray*}
\left\Vert \mathsf{T}_{b,d}^{\left(0,1,1,0\right) }\right\Vert _{L^2\left(%
\mathcal{H}\right)\rightarrow L^2\left(\mathcal{H}\right)} & \approx & C_1 \\
\left\Vert \mathsf{T}_{d,b}^{\left( 0,1,1,0\right) }\right\Vert _{L^2\left(%
\mathcal{H}\right)\rightarrow L^2\left(\mathcal{H}\right)} & \approx & C_2
\end{eqnarray*}
with $C_1$ and $C_2$ the best constants in the testing inequality. However,
for each of these constants we have 
\begin{equation*}
C_j\leq \left\Vert \mathsf{P}_{b}^{\left( 0,1\right) }\circ \mathsf{P}%
_{d}^{\left(1,0\right) }\right\Vert _{L^{2}\left( \mathbb{R}\right)
\rightarrow L^{2}\left( \mathbb{R}\right)},
\end{equation*}
see the argument just after \eqref{par test}, and so we obtain 
\begin{equation*}
\left\Vert \mathsf{P}_{b}^{\left( 0,1\right) }\circ \mathsf{P}_{d}^{\left(
1,0\right) }\right\Vert _{L^{2}\left( \mathbb{R}\right) \rightarrow
L^{2}\left( \mathbb{R}\right) }\approx \left\Vert \mathsf{T}_{b,d}^{\left(
0,1,1,0\right) }\right\Vert _{L^2\left(\mathcal{H}\right)\rightarrow
L^2\left(\mathcal{H}\right)}+\left\Vert \mathsf{T}_{d,b}^{\left(
0,1,1,0\right) }\right\Vert _{L^2\left(\mathcal{H}\right)\rightarrow
L^2\left(\mathcal{H}\right)}\ .
\end{equation*}

Now the operator norm $\left\Vert \mathsf{T}_{b,d}^{\left( 0,1,1,0\right)
}\right\Vert _{L^{2}\left( \mathcal{H}\right) \rightarrow L^{2}\left( 
\mathcal{H}\right) }$ equals the best constant in a certain two weight
inequality for the positive operator $\mathsf{U}$ on $L_{c}^{2}\left( 
\mathcal{H}\right) $, where%
\begin{equation*}
\mathsf{U}\equiv \sum_{K\in \mathcal{D}}\widetilde{\mathbf{1}}_{T\left(
K\right) }\otimes \widetilde{\mathbf{1}}_{Q\left( K\right) }.
\end{equation*}%
The inequality we wish to characterize is%
\begin{equation}
\left\Vert \mathcal{M}_{\overline{b}}^{0}\mathsf{U}\mathcal{M}%
_{d}^{-1}f\right\Vert _{L_{c}^{2}\left( \mathcal{H}\right) }=\left\Vert 
\mathsf{T}_{b,d}^{\left( 0,1,1,0\right) }f\right\Vert _{L_{c}^{2}\left( 
\mathcal{H}\right) }\lesssim \left\Vert f\right\Vert _{L_{c}^{2}\left( 
\mathcal{H}\right) },  \label{to be char1}
\end{equation}%
which we first recast in the language of trees as in \cite{ArRoSa}. To do this, we suppose that $f$ is constant on tiles $T\left(
K\right) $ in the upper half space, and view $f$ as the sequence $f:\mathcal{%
D}\longrightarrow \mathbb{C}$ given by its averages%
\begin{equation*}
f\left( K\right) \equiv \left\langle \frac{1}{\left\vert T\left( K\right)
\right\vert }\mathbf{1}_{T\left( K\right) },f\right\rangle_{L^2(\mathcal{H})} .
\end{equation*}%
Define the adjoint tree integral $\mathcal{I}^{\ast }f$ by 
\begin{equation*}
\mathcal{I}^{\ast }f\left( K\right) \equiv\sum\limits_{L\in \mathcal{D}:\
L\subset K}f\left( L\right) ,\ \ \ \ \ K\in \mathcal{D},
\end{equation*}%
and define the special weight sequence $s\left( K\right) \equiv \left\vert
K\right\vert $, $K\in \mathcal{D}$. Then for $f$ constant on tiles $T\left(
K\right) $ in the upper half space we have%
\begin{eqnarray*}
\mathsf{U}f &=&\sum_{K\in \mathcal{D}}\widetilde{\mathbf{1}}_{T\left(
K\right) }\otimes \widetilde{\mathbf{1}}_{Q\left( K\right) }f=\sum_{K\in 
\mathcal{D}}\left\langle \widetilde{\mathbf{1}}_{Q\left( K\right)
},f\right\rangle_{L^2(\mathcal{H})} \widetilde{\mathbf{1}}_{T\left( K\right) } \\
&=&\sum_{K\in \mathcal{D}}\left\langle \frac{1}{\sqrt{\left\vert Q\left(
K\right) \right\vert }}\sum\limits_{L\subset K}\mathbf{1}_{T\left( L\right)
},f\right\rangle_{L^2(\mathcal{H})} \frac{1}{\sqrt{\left\vert T\left( K\right) \right\vert }}%
\mathbf{1}_{T\left( K\right) } \\
&=&\sum_{K\in \mathcal{D}}\frac{1}{\sqrt{\left\vert Q\left( K\right)
\right\vert }}\sum\limits_{L\subset K}\left\vert T\left( L\right)
\right\vert \left\langle \frac{1}{\left\vert T\left( L\right) \right\vert }%
\mathbf{1}_{T\left( L\right) },f\right\rangle_{L^2(\mathcal{H})} \frac{1}{\sqrt{\left\vert
T\left( K\right) \right\vert }}\mathbf{1}_{T\left( K\right) } \\
&=&\frac{1}{\sqrt{2}}\sum_{K\in \mathcal{D}}\left\{ \sum\limits_{L\subset K}%
\frac{1}{2}s\left( L\right) ^{2}f\left( L\right) \right\} \frac{1}{\frac{1}{2%
}\left\vert K\right\vert ^{2}}\mathbf{1}_{T\left( K\right) } \\
&=&\sqrt{2}\sum_{K\in \mathcal{D}}\mathcal{I}^{\ast }\left( s^{2}f\right)
\left( K\right) \frac{1}{s\left( K\right) ^{2}}\mathbf{1}_{T\left( K\right)
}\ ,
\end{eqnarray*}%
which shows that%
\begin{equation*}
\left( \mathsf{U}f\right) \left( K\right) =\sqrt{2}\frac{1}{s\left( K\right)
^{2}}\mathcal{I}^{\ast }\left( s^{2}f\right) \left( K\right) .
\end{equation*}%
Since $\mathcal{M}_{d}^{-1}$ and $\mathcal{M}_{\overline{b}}^{0}$ are
multiplication by $\frac{d_{K}}{\left\vert K\right\vert }=\frac{d_{K}}{%
s\left( K\right) }$ and $\overline{b_{K}}$ respectively on the tile $T\left(
K\right) $, which for convenience we abbreviate as $\frac{d}{s}$ and $%
\overline{b}$ respectively, we see that the two weight inequality (\ref{to
be char1}) is equivalent to%
\begin{equation}
\left\Vert s\left\vert \overline{b}\right\vert \frac{1}{s^{2}}\mathcal{I}%
^{\ast }\left( s^{2}\frac{\left\vert d\right\vert }{s}f\right) \right\Vert
_{\ell ^{2}\left( \mathcal{D}\right) }\lesssim \left\Vert sf\right\Vert
_{\ell ^{2}\left( \mathcal{D}\right) }.  \label{to be char1'}
\end{equation}%
Now if we set 
\begin{eqnarray*}
s\left\vert d\right\vert f &=&g\omega , \\
\left( sf\right) ^{2} &=&g^{2}\omega , \\
\left( \frac{\left\vert b\right\vert }{s}\right) ^{2} &=&\sigma ,
\end{eqnarray*}%
then 
\begin{equation*}
\omega =\frac{\left( g\omega \right) ^{2}}{g^{2}\omega }=\frac{\left(
s\left\vert d\right\vert f\right) ^{2}}{\left( sf\right) ^{2}}=d^{2},
\end{equation*}%
and (\ref{to be char1'}) is equivalent to%
\begin{equation}
\left\Vert \mathcal{I}^{\ast }\left( g\omega \right) \right\Vert _{\ell
^{2}\left( \sigma \right) }\leq C\left\Vert g\right\Vert _{\ell ^{2}\left(
\omega \right) }.  \label{to be char1''}
\end{equation}

At this point we can apply the characterization of the two weight tree
inequality in \cite{ArRoSa}. Now $\mathcal{D}$ is a rootless
tree, and the inequality in \cite{ArRoSa} is stated for a rooted tree, but the
monotone convergence theorem immediately extends the characterization in
\cite{ArRoSa} to rootless trees as well. Thus the best constant $C$ in (\ref{to
be char1''}) is comparable to the best constant $C_{1}$ in the corresponding
truncated testing condition with $g=\mathbf{1}_{\left\{ L\in D:\ L\subset
I\right\} }$ for $I\in \mathcal{D}$: 
\begin{eqnarray*}
\sum_{J\subset I}\left( \sum_{L\subset J}\omega \left( L\right) \right)
^{2}\sigma \left( J\right)  &=&\sum_{J\subset I}\left[ \mathcal{I}^{\ast
}\omega \left( J\right) \right] ^{2}\sigma \left( J\right) \leq \left\Vert 
\mathcal{I}^{\ast }\left( g\omega \right) \right\Vert _{\ell ^{2}\left(
\sigma \right) }^{2} \\
&\leq &C_{1}^{2}\left\Vert g\right\Vert _{\ell ^{2}\left( \omega \right)
}=C_{1}^{2}\mathcal{I}^{\ast }\omega \left( I\right)
=C_{1}^{2}\sum_{L\subset I}\omega \left( L\right) ,\ \ \ \ \ I\in \mathcal{D}%
,
\end{eqnarray*}%
i.e.%
\begin{equation}
\sum_{J\subset I}\left( \sum_{L\subset J}\left\vert d_{L}\right\vert
^{2}\right) ^{2}\frac{\left\vert b_{J}\right\vert ^{2}}{\left\vert
J\right\vert ^{2}}\leq C_{1}^{2}\sum_{L\subset K}\left\vert d_{L}\right\vert
^{2},\ \ \ \ \ K\in \mathcal{D}.  \label{testing condition}
\end{equation}

It is now convenient to relabel our weights by introducing the \emph{%
different} notation,

\begin{eqnarray*}
w &\equiv &\sum_{I\in \mathcal{D}}\left\vert b_{I}\right\vert ^{2}\mathbf{1}%
_{T\left( I\right) } \\
\sigma  &\equiv &\sum_{I\in \mathcal{D}}\frac{\left\vert d_{I}\right\vert
^{2}}{\left\vert I\right\vert ^{2}}\mathbf{1}_{T\left( I\right) },
\end{eqnarray*}%
in which the testing condition (\ref{testing condition}) is equivalent to 
\begin{equation}
\left\Vert \mathbf{1}_{Q(I)}\mathsf{U}\left( \sigma \mathbf{1}_{Q(I)}\right)
\right\Vert _{L^{2}\left( \mathcal{H};w\right) }^{2}\leq C_{1}^{2}\left\Vert 
\mathbf{1}_{Q(I)}\right\Vert _{L^{2}\left( \mathcal{H};\sigma \right) }^{2},
\label{postive_testing}
\end{equation}%
since 
\begin{eqnarray*}
\left\Vert \mathbf{1}_{Q(I)}\mathsf{U}\left( \sigma \mathbf{1}_{Q(I)}\right)
\right\Vert _{L^{2}\left( \mathcal{H};w\right) }^{2} &=&2\left\Vert
\sum_{J\subset I}\frac{\left\langle \sigma \mathbf{1}_{Q(I)},\widetilde{%
\mathbf{1}}_{Q(J)}\right\rangle _{L^{2}\left( \mathcal{H}\right) }}{%
\left\vert J\right\vert }\mathbf{1}_{T(J)}\right\Vert _{L^{2}\left( \mathcal{%
H};w\right) }^{2} \\
&=&2\sum_{J\subset I}\left\vert b_{I}\right\vert ^{2}\left( \left\langle
\sigma \mathbf{1}_{Q(I)},\widetilde{\mathbf{1}}_{Q(J)}\right\rangle
_{L^{2}\left( \mathcal{H}\right) }\right) ^{2} \\
&=&2\sum_{J\subset I}\left\vert b_{I}\right\vert ^{2}\frac{1}{\left\vert
J\right\vert ^{2}}\left( \sum_{L\in \mathcal{D}}\frac{\left\vert
d_{L}\right\vert ^{2}}{\left\vert L\right\vert ^{2}}\int_{\mathcal{H}}%
\mathbf{1}_{T(L)}\mathbf{1}_{Q(I)}\mathbf{1}_{Q(J)}\,dA\right) ^{2} \\
&=&\frac{1}{2}\sum_{J\subset I}\left\vert b_{I}\right\vert ^{2}\frac{1}{%
\left\vert J\right\vert ^{2}}\left( \sum_{L\subset J}\left\vert
d_{L}\right\vert ^{2}\right) ^{2}.
\end{eqnarray*}

and

\begin{equation}  \label{e.rhs}
\left\Vert \mathsf{Q}_{I}d\right\Vert _{L^{2}(\mathbb{R})}^{2}=\sum_{L%
\subset I}\left\vert d_{L}\right\vert^{2}=\left\Vert
1_{Q(I)}\right\Vert_{L^2\left(\mathcal{H};\sigma\right)}^2.
\end{equation}

The testing condition \eqref{postive_testing} thus gives a characterization
of the boundedness of the paraproduct composition $\mathsf{P}_{b}^{\left(
0,1\right) }\circ \mathsf{P}_{d}^{\left( 1,0\right) }$ on $L^{2}\left( 
\mathbb{R}\right) $. However, we now want to rephrase this as a testing
condition, but only on the operator $\mathsf{P}_{b}^{\left( 0,1\right)
}\circ \mathsf{P}_{d}^{\left( 1,0\right) }$.


Let $V:L^2\left(\mathbb{R}\right)\rightarrow L^2_{c}\left(\mathcal{H}\right)$
be the unitary operator defined on basis elements by 
\begin{equation}  \label{e.changeofbase}
V h_I=\widetilde{\mathbf{1}}_{T(I)}.
\end{equation}
Using these unitary operators we can write

\begin{eqnarray*}
\left\Vert \mathsf{Q}_I \mathsf{P}_{b}^{(0,1)}\mathsf{P}_{d}^{(1,0)} \mathsf{%
Q}_I \overline{d} \right\Vert_{L^2(\mathbb{R})} & = & \left\Vert \mathsf{Q}%
_I V^\ast \left(T^{(0,1,1,0)}_{b,d}\right)V\mathsf{Q}_I \overline{d}
\right\Vert_{L^2(\mathbb{R})} \\
& = & \left\Vert \mathsf{Q}_I V^\ast \left(T^{(0,1,1,0)}_{b,d}\right)
\left(\sum_{J\subset I} \overline{d_J}\, \widetilde{\mathbf{1}}%
_{T(J)}\right) \right\Vert_{L^2(\mathbb{R})} \\
& = & \left\Vert\mathsf{Q}_I V^\ast \mathcal{M}_{\overline{b}}^{0}\mathsf{U}%
\mathcal{M}_{d}^{-1}\left(\sum_{J\subset I} \overline{d_J}\, \widetilde{%
\mathbf{1}}_{T(J)}\right) \right\Vert_{L^2(\mathbb{R})} \\
& = & \left\Vert \mathsf{Q}_I V^\ast \mathcal{M}_{\overline{b}}^{0}\mathsf{U}
\left( \sigma\mathbf{1}_{Q(I)}\right) \right\Vert_{L^2(\mathbb{R})} \\
& = & \left\Vert V^\ast \mathbf{1}_{Q(I)} \mathcal{M}_{\overline{b}}^{0}%
\mathsf{U}\left( \sigma\mathbf{1}_{Q(I)}\right) \right\Vert_{L^2(\mathbb{R})}
\\
& = & \left\Vert \mathbf{1}_{Q(I)} \mathcal{M}_{\overline{b}}^{0}\mathsf{U}%
\left( \sigma\mathbf{1}_{Q(I)}\right) \right\Vert_{L^2_c(\mathcal{H})} \\
& = & \left\Vert \mathbf{1}_{Q(I)} \mathsf{U}\left( \sigma\mathbf{1}%
_{Q(I)}\right) \right\Vert_{L^2_c(\mathcal{H};w)}.
\end{eqnarray*}

Finally, using \eqref{e.rhs} one sees that \eqref{postive_testing} is
equivalent to a simple testing condition on the composition $\mathsf{P}%
_{b}^{\left( 0,1\right) }\circ \mathsf{P}_{d}^{\left( 1,0\right) }$:%
\begin{equation}
\left\Vert \mathsf{Q}_I \mathsf{P}_{b}^{(0,1)}\mathsf{P}_{d}^{(1,0)} \mathsf{%
Q}_I \overline{d} \right\Vert_{L^2(\mathbb{R})} \lesssim \left\Vert \mathsf{Q%
}_{I}d\right\Vert _{L^{2}(\mathbb{R})}^{2}.  \label{par test}
\end{equation}

Furthermore, it is immediate to see that \eqref{par test} is implied by the
boundedness of $\mathsf{P}_{b}^{\left( 0,1\right) }\circ \mathsf{P}%
_{d}^{\left( 1,0\right) }$ on $L^2\left(\mathbb{R}\right)$. Interchanging
the roles of $b$ and $d$ we have the following Theorem that characterizes
the boundedness of $\mathsf{P}_{b}^{\left( 0,1\right) }\circ \mathsf{P}%
_{d}^{\left( 1,0\right) }$, which is just a restatement of Theorem \ref%
{Theorem_Positive}.

\begin{theorem}
The composition $\mathsf{P}_{b}^{\left( 0,1\right) }\circ \mathsf{P}%
_{d}^{\left( 1,0\right) }$ is bounded on $L^{2}\left( \mathbb{R}\right) $ if
and only if both%
\begin{eqnarray*}
\left\Vert \mathsf{Q}_{I}\mathsf{P}_{b}^{\left( 0,1\right) }\mathsf{P}%
_{d}^{\left( 1,0\right) }\left( \mathsf{Q}_{I}\overline{d}\right)
\right\Vert _{L^{2}(\mathbb{R})}^{2} & \leq & C_1^2 \left\Vert \mathsf{Q}%
_{I}d\right\Vert _{L^{2}(\mathbb{R})}^{2}; \\
\left\Vert \mathsf{Q}_{I}\mathsf{P}_{d}^{\left( 0,1\right) }\mathsf{P}%
_{b}^{\left( 1,0\right) }\left( \mathsf{Q}_{I}\overline{b}\right)
\right\Vert _{L^{2}(\mathbb{R})}^{2} & \leq & C_2^2\left\Vert \mathsf{Q}%
_{I}b\right\Vert _{L^{2}(\mathbb{R})}^{2}
\end{eqnarray*}
for all $I\in\mathcal{D}$; i.e. for all $I\in\mathcal{D}$ the following
inequalities are true 
\begin{eqnarray*}
\sum_{J\subset I}\left\vert b_{J}\right\vert^{2}\frac{1}{\left\vert
J\right\vert ^{2}}\left( \sum_{L\subset J}\left\vert
d_{L}\right\vert^{2}\right) ^{2} &\leq & C_1^2\sum_{L\subset I}\left\vert
d_{L}\right\vert^{2}; \\
\sum_{J\subset I}\left\vert d_{J}\right\vert^{2}\frac{1}{\left\vert
J\right\vert ^{2}}\left( \sum_{L\subset J}\left\vert
b_{L}\right\vert^{2}\right) ^{2} & \leq & C_2^2 \sum_{L\subset I}\left\vert
b_{L}\right\vert^{2}.
\end{eqnarray*}%
Moreover, the norm of $\mathsf{P}_{b}^{\left( 0,1\right) }\circ \mathsf{P}%
_{d}^{\left( 1,0\right) }$ on $L^2\left(\mathbb{R}\right)$ satisfies 
\begin{equation*}
\left\Vert \mathsf{P}_{b}^{\left( 0,1\right) }\circ \mathsf{P}_{d}^{\left(
1,0\right) }\right\Vert _{L^{2}\left( \mathbb{R}\right) \rightarrow
L^{2}\left( \mathbb{R}\right) }\approx C_1+C_2
\end{equation*}
where $C_1$ and $C_2$ are the best constants in appearing above.
\end{theorem}

\subsection{Type (0,\,1,\,0,\,0) Compositions}

\label{0100}

The Gram matrix $\mathfrak{G}_{\mathsf{P}_{b}^{\left( 0,1\right) }\circ 
\mathsf{P}_{d}^{\left( 0,0\right) }}=\left[ G_{I,J}\right] _{I,J\in \mathcal{%
D}}$ of the operator $\mathsf{P}_{b}^{\left( 0,1\right) }\circ \mathsf{P}%
_{d}^{\left( 0,0\right) }$ relative to the Haar basis $\left\{ h_{I}\right\}
_{I\in \mathcal{D}}$ has entries given by%
\begin{eqnarray*}
G_{I,J} &=&\left\langle \mathsf{P}_{b}^{\left( 0,1\right) }\circ \mathsf{P}%
_{d}^{\left( 0,0\right) }h_{J},h_{I}\right\rangle_{L^2(\mathbb{R})}
=\left\langle \mathsf{P}_{d}^{\left( 0,0\right) }h_{J},\mathsf{P}%
_{b}^{\left( 1,0\right) }h_{I}\right\rangle_{L^2(\mathbb{R})} \\
&=&\left\langle d_{J}h_{J},b_{I}h_{I}^1\right\rangle_{L^2(\mathbb{R})} \\
&=&\overline{b_{I}}d_{J}\widehat{h_{I}^1}\left( J\right) =\left\{ 
\begin{array}{ccc}
\overline{b_{I}}d_{J}\frac{-1}{\sqrt{\left\vert J\right\vert }} & \text{ if }
& I\subset J_{-} \\ 
\overline{b_{I}}d_{J}\frac{1}{\sqrt{\left\vert J\right\vert }} & \text{ if }
& I\subset J_{+} \\ 
0 & \text{ if } & J\subset I\text{ or }I\cap J=\emptyset.%
\end{array}%
\right.
\end{eqnarray*}

Now consider the operator $\mathsf{T}_{b,d}^{\left( 0,1,0,0\right) }$
defined by%
\begin{equation*}
\mathsf{T}_{b,d}^{\left( 0,1,0,0\right) }\equiv \mathcal{M}_{\overline{b}%
}^{-1}\left( \sum_{K\in \mathcal{D}}\widetilde{\mathbf{1}}_{Q_{\pm }\left(
K\right) }\otimes \widetilde{\mathbf{1}}_{T\left( K\right) }\right) \mathcal{%
M}_{d}^{\frac{1}{2}},
\end{equation*}
where 
\begin{equation}  \label{Qpm_Haar}
\mathbf{1}_{Q_{\pm }\left( K\right) }\equiv -\sum_{L\subset K_{-}}\mathbf{1}%
_{T\left( L\right) }+\sum_{L\subset K_{+}}\mathbf{1}_{T\left( L\right) }.
\end{equation}
A straightforward computation shows that 
\begin{eqnarray*}
\left\Vert \mathbf{1}_{Q_{\pm}\left (K\right)}\right\Vert_{L^2\left(\mathcal{%
H}\right)} & = & \frac{\left\vert K\right\vert}{2}; \\
\mathcal{M}_{a}^{\lambda}\mathbf{1}_{Q_{\pm}\left( K\right)} & = &
-\sum_{L\subset K_{-}}a_{L}\left\vert L\right\vert^{\lambda}\mathbf{1}%
_{T\left( L\right) }+\sum_{L\subset K_{+}}a_{L}\left\vert
L\right\vert^{\lambda}\mathbf{1}_{T\left( L\right) }.
\end{eqnarray*}

The Gram matrix $\mathfrak{G}_{\mathsf{T}_{b,d}^{\left( 0,1,0,0\right) }}=%
\left[ G_{I,J}\right] _{I,J\in \mathcal{D}}$ of $\mathsf{T}_{b,d}^{\left(
0,1,0,0\right) }$ relative to the basis $\left\{ \widetilde{\mathbf{1}}%
_{T\left( I\right) }\right\} _{I\in \mathcal{D}}$ then has entries given by%
\begin{eqnarray*}
G_{I,J} &=&\left\langle \mathsf{T}_{b,d}^{\left( 0,1,0,0\right) }\widetilde{%
\mathbf{1}}_{T\left( J\right) },\widetilde{\mathbf{1}}_{T\left( I\right)
}\right\rangle_{L^2\left(\mathcal{H}\right)} \\
&=&\left\langle \mathcal{M}_{\overline{b}}^{-1}\left( \sum_{K\in \mathcal{D}}%
\widetilde{\mathbf{1}}_{Q_{\pm }\left( K\right) }\otimes \widetilde{\mathbf{1%
}}_{T\left( K\right) }\right) \mathcal{M}_{d}^{\frac{1}{2}}\widetilde{%
\mathbf{1}}_{T\left( J\right) }, \widetilde{\mathbf{1}}_{T\left( I\right)
}\right\rangle_{L^2\left(\mathcal{H}\right)} \\
&=&\sum_{K\in \mathcal{D}}\left\langle \left\langle \widetilde{\mathbf{1}}%
_{T\left( K\right) },\mathcal{M}_{d}^{\frac{1}{2}}\widetilde{\mathbf{1}}%
_{T\left( J\right) }\right\rangle_{L^2\left(\mathcal{H}\right)} \mathcal{M}_{%
\overline{b}}^{-1}\widetilde{\mathbf{1}}_{Q_{\pm }\left( K\right) },%
\widetilde{\mathbf{1}}_{T\left( I\right) }\right\rangle_{L^2\left(\mathcal{H}%
\right)} \\
&=&\sum_{K\in \mathcal{D}}d_{J}\left\vert J\right\vert ^{\frac{1}{2}%
}\left\langle \widetilde{\mathbf{1}}_{T\left( K\right) },\widetilde{\mathbf{1%
}}_{T\left( J\right) }\right\rangle_{L^2\left(\mathcal{H}\right)}
\left\langle \mathcal{M}_{\overline{b}}^{-1}\widetilde{\mathbf{1}}_{Q_{\pm
}\left( K\right) },\widetilde{\mathbf{1}}_{T\left( I\right)
}\right\rangle_{L^2\left(\mathcal{H}\right)} \\
&=& d_{J}\left\vert J\right\vert ^{\frac{1}{2}}\left\langle \mathcal{M}_{%
\overline{b}}^{-1}\widetilde{\mathbf{1}}_{Q_{\pm }\left( J\right) },%
\widetilde{\mathbf{1}}_{T\left( I\right) }\right\rangle_{L^2\left(\mathcal{H}%
\right)} \\
&=& 2\sqrt{2} d_{J}\left\vert J\right\vert ^{\frac{1}{2}}\left\vert
J\right\vert^{-1} \left\vert I\right\vert^{-1}\left\langle -\sum_{L\subset
J_{-}}\overline{b_{L}}\left\vert L\right\vert^{-1} \mathbf{1}_{T\left(
L\right) }+\sum_{L\subset J_{+}}\overline{b_{L}}\left\vert L\right\vert^{-1}%
\mathbf{1}_{T\left( L\right) }, \mathbf{1}_{T\left( I\right)
}\right\rangle_{L^2\left(\mathcal{H}\right)} \\
&=& \sqrt{2}\left\{ 
\begin{array}{ccc}
-\overline{b_{I}}d_{J}\left\vert J\right\vert ^{-\frac{1}{2}} & \text{ if }
& I\subset J_{-} \\ 
\overline{b_{I}}d_{J}\left\vert J\right\vert ^{-\frac{1}{2}} & \text{ if } & 
I\subset J_{+} \\ 
0 & \text{ if } & J\subset I\text{ or }I\cap J=\emptyset.%
\end{array}%
\right.
\end{eqnarray*}%
Thus, up to an absolute constant, we see that $\mathfrak{G}_{\mathsf{T}%
_{b,d}^{\left( 0,1,0,0\right) }}=\mathfrak{G}_{\mathsf{P}_{b}^{\left(
0,1\right) }\circ \mathsf{P}_{d}^{\left( 0,0\right) }}$ , and we obtain the
following conclusion 
\begin{equation*}
\left\Vert \mathsf{P}_{b}^{\left( 0,1\right) }\circ \mathsf{P}_{d}^{\left(
0,0\right) }\right\Vert _{L^{2}\left( \mathbb{R}\right) \rightarrow
L^{2}\left( \mathbb{R}\right) }\approx \left\Vert \mathsf{T}_{b,d}^{\left(
0,1,0,0\right) }\right\Vert _{L^2\left(\mathcal{H}\right)\rightarrow
L^2\left(\mathcal{H}\right)}\ .
\end{equation*}

\bigskip

Now the operator norm $\left\Vert \mathsf{T}_{b,d}^{\left( 0,1,0,0\right)
}\right\Vert _{L^2_c\left(\mathcal{H}\right)\rightarrow L^2_c\left(\mathcal{H%
}\right)}$ equals the best constant in a certain two weight inequality for
the operator $\mathsf{U}$ on $L^2\left(\mathcal{H}\right)$ defined by%
\begin{equation*}
\mathsf{U}\equiv \sum_{K\in \mathcal{D}}\widetilde{\mathbf{1}}_{Q_{\pm
}\left( K\right) }\otimes \widetilde{\mathbf{1}}_{T\left( K\right) }.
\end{equation*}%
This operator is not positive, but its singular character is well-behaved,
and the best constant in a certain two weight inequality associated to $%
\mathsf{U}$ is in turn comparable to the best constants in the associated
testing conditions. These testing conditions thus give a characterization of
the boundedness of the paraproduct composition $\mathsf{P}_{b}^{\left(
0,1\right) }\circ \mathsf{P}_{d}^{\left( 0,0\right) }$ on $L^{2}\left( 
\mathbb{R}\right) $. Here are the details which provide this reduction.

By the computations above, the inequality we wish to characterize is:%
\begin{equation}
\left\Vert \mathcal{M}_{\overline{b}}^{-1}\mathsf{U} \mathcal{M}_{d}^{\frac{1%
}{2}}f\right\Vert _{L^2_c\left(\mathcal{H}\right)} = \left\Vert \mathsf{T}%
_{b,d}^{\left( 0,1,0,0\right) }f\right\Vert _{L^2_c\left(\mathcal{H}\right)}
\lesssim \left\Vert f\right\Vert _{L^2_c\left(\mathcal{H}\right)}.
\label{to be char}
\end{equation}
Now if $f=\sum_{I\in\mathcal{D}} f_I\mathbf{1}_{T(I)}$ then $\mathcal{M}%
_{d}^{\frac{1}{2}}f=\sum_{I\in\mathcal{D}}d_{I}\sqrt{\left\vert I\right\vert 
}f_I\mathbf{1}_{T\left( I\right)}$, and if we define $g=\mathcal{M}_d^{\frac{%
1}{2}}f$ then inequality \eqref{to be char} is equivalent to:%
\begin{equation}
\left\Vert \mathsf{U}g\right\Vert _{L^2_c\left(\mathcal{H};w\right)}\lesssim
\left\Vert g\right\Vert _{L^2_c\left(\mathcal{H};\nu\right)} ,
\label{to be char-redux}
\end{equation}%
where the weights $w$ and $\nu $ are given by%
\begin{eqnarray*}
w & \equiv & \sum_{I\in\mathcal{D}}\left\vert b_{I}\right\vert^{2}\left\vert
I\right\vert ^{-2}\mathbf{1}_{T(I)} \\
\nu & \equiv & \sum_{I\in\mathcal{D}} \left\vert
d_{I}\right\vert^{-2}\left\vert I\right\vert ^{-1}\mathbf{1}_{T(I)}.
\label{def weights}
\end{eqnarray*}
This follows because a straightforward computation shows that for $%
k=\sum_{I\in\mathcal{D}}k_I\mathbf{1}_{T(I)}$ 
\begin{equation*}
\left\Vert \mathcal{M}_{\overline{b}}^{1} k\right\Vert_{L^2_c\left(\mathcal{H%
}\right)}^2=\frac{1}{2}\sum_{I\in\mathcal{D}}\left\vert b_I\right\vert^2
\left\vert k_I\right\vert^2\left\vert I\right\vert^{4}=\left\Vert
k\right\Vert_{L^2_c\left(\mathcal{H};w\right)}^2
\end{equation*}
and that for $f=\left(\mathcal{M}_d^{\frac{1}{2}}\right)^{-1}g=\sum_{I\in%
\mathcal{D}} g_I d_I^{-1} \left\vert I\right\vert^{-\frac{1}{2}}\mathbf{1}%
_{T(I)}$ one has 
\begin{equation*}
\left\Vert f\right\Vert_{L^2_c\left(\mathcal{H}\right)}^2=\frac{1}{2}%
\sum_{I\in\mathcal{D}} \left\vert g_I\right\vert^2\left\vert
d_I\right\vert^{-2}\left\vert I\right\vert=\left\Vert
g\right\Vert_{L^2_c\left(\mathcal{H};\nu\right)}^2.
\end{equation*}
Now, let 
\begin{equation*}
\sigma\equiv \sum_{I\in\mathcal{D}} \left\vert
d_{I}\right\vert^{2}\left\vert I\right\vert \mathbf{1}_{T(I)}
\end{equation*}
and substitute $g=h\sigma$ into \eqref{to be char-redux} to see that %
\eqref{to be char} is in terms of weighted $L^2$ norms equivalent to 
\begin{equation}
\left\Vert \mathsf{U}\left(\sigma h\right)\right\Vert_{L^2_c\left(\mathcal{H}%
;w\right)}\lesssim \left\Vert h\right\Vert_{L^2_c\left(\mathcal{H}%
;\sigma\right)} ,  \label{to be char-final}
\end{equation}%
By Theorem \ref{NTV_Basis} we have that the best constant in 
\eqref{to be
char-final} is equivalent to best constants in the testing conditions given
by 
\begin{eqnarray}
\left\Vert \mathsf{U}\left(\sigma \mathbf{1}_{T\left( I\right) }\right)
\right\Vert_{L^2_c\left(\mathcal{H};w\right)} &\leq & C_1 \left\Vert \mathbf{%
1}_{T\left( I\right)}\right\Vert_{L^2_c\left(\mathcal{H};\sigma\right)}
\label{NTV_forward} \\
\left\Vert \mathbf{1}_{Q\left( I\right)} \mathsf{U}^{\ast}\left(w \mathbf{1}%
_{Q\left( I\right) }\right) \right\Vert_{L^2_c\left(\mathcal{H}%
;\sigma\right)} &\leq & C_2 \left\Vert \mathbf{1}_{Q\left(I\right)}\right%
\Vert_{L^2_c\left(\mathcal{H};w\right)} .  \label{NTV_backward}
\end{eqnarray}

We now phrase these conditions in terms of paraproduct type testing
conditions. For our special choice of measures $\sigma$ and $w$ given above,
a simple computation shows that right-hand side of \eqref{NTV_forward} is 
\begin{equation*}
\left\Vert\mathbf{1}_{T\left( I\right)}\right\Vert_{L^2_c\left(\mathcal{H}%
;\sigma\right)}=\frac{\left\vert d_I\right\vert \left\vert I\right\vert^{%
\frac{3}{2}}}{\sqrt{2}}.
\end{equation*}
While the left-hand side of \eqref{NTV_forward} yields 
\begin{eqnarray*}
\left\Vert \mathsf{U}\left(\sigma \mathbf{1}_{T\left( I\right) }\right)
\right\Vert_{L^2_c\left(\mathcal{H};w\right)} & = & \sqrt{2}\left\vert
I\right\vert\left\vert d_I\right\vert^2 \left\Vert \widetilde{\mathbf{1}}%
_{Q_{\pm}\left(I\right)} \right\Vert_{L^2_c\left(\mathcal{H};w\right)} \\
& = & \left\vert I\right\vert\left\vert d_I\right\vert^2
\left(\sum_{L\subset I_{-}} \left\vert b_L\right\vert^2+\sum_{L\subset
I_{+}} \left\vert b_L\right\vert^2\right)^{\frac{1}{2}} \\
& = & \left\vert I\right\vert\left\vert
d_I\right\vert^2\left(\sum_{L\subsetneq I} \left\vert
b_L\right\vert^2\right)^{\frac{1}{2}}.
\end{eqnarray*}
Thus, for our special choice of measures, \eqref{NTV_forward} is equivalent
to 
\begin{equation*}
\sup_{I\in\mathcal{D}}\left\vert d_I\right\vert \left(\frac{1}{\left\vert
I\right\vert}\sum_{L\subsetneq I} \left\vert b_L\right\vert^2\right)^{\frac{1%
}{2}}\lesssim 1.
\end{equation*}
And, since $\left(\frac{1}{\left\vert I\right\vert}\sum_{L\subsetneq I}
\left\vert b_L\right\vert^2\right)^{\frac{1}{2}}=\left\Vert \mathsf{P}%
_{b}^{\left( 0,1\right) } h_I\right\Vert_{L^2\left(\mathbb{R}\right)}$ we
conclude that \eqref{NTV_forward} is implied by 
\begin{equation*}
\left\vert d_I\right\vert\left\Vert \mathsf{P}_{b}^{\left( 0,1\right) }
h_I\right\Vert_{L^2\left(\mathbb{R}\right)}\lesssim 1\quad\forall I\in%
\mathcal{D}.
\end{equation*}
Furthermore, it is clear that this last condition is implied by the
boundedness of the operator $\mathsf{P}_{b}^{\left( 0,1\right) }\circ 
\mathsf{P}_{d}^{\left( 0,0\right) }$ from $L^2\left(\mathbb{R}%
\right)\rightarrow L^2\left(\mathbb{R}\right)$. For if $g$ is any function,
then we have 
\begin{eqnarray*}
\left\Vert g\right\Vert_{L^2\left(\mathbb{R}\right)}\left\Vert \mathsf{P}%
_{b}^{\left( 0,1\right) }\circ \mathsf{P}_{d}^{\left( 0,0\right)
}\right\Vert_{L^2\left (\mathbb{R}\right)\rightarrow L^2\left (\mathbb{R}%
\right)} & \geq & \left\vert \left\langle \mathsf{P}_{b}^{\left( 0,1\right)
}\circ \mathsf{P}_{d}^{\left( 0,0\right) } h_I,g\right\rangle_{L^2\left (%
\mathbb{R}\right)}\right\vert \\
& = & \left\vert d_I\right\vert \left\vert \left\langle \mathsf{P}%
_{b}^{\left( 0,1\right) } h_I,g\right\rangle_{L^2\left (\mathbb{R}%
\right)}\right\vert.
\end{eqnarray*}
Choosing $g=\mathsf{P}_{b}^{\left( 0,1\right) } h_I$ yields, 
\begin{equation*}
\left\vert d_I\right\vert\left(\frac{1}{\left\vert I\right\vert}%
\sum_{L\subsetneq I} \left\vert b_L\right\vert^2\right)^{\frac{1}{2}%
}=\left\vert d_I\right\vert\left\Vert \mathsf{P}_{b}^{\left( 0,1\right) }
h_I\right\Vert_{L^2\left(\mathbb{R}\right)}\leq\left\Vert \mathsf{P}%
_{b}^{\left( 0,1\right) }\circ \mathsf{P}_{d}^{\left( 0,0\right)
}\right\Vert_{L^2\left (\mathbb{R}\right)\rightarrow L^2\left (\mathbb{R}%
\right)}.
\end{equation*}

Turning to \eqref{NTV_backward} one easily computes that the right-hand side
is given by 
\begin{equation*}
\left\Vert\mathbf{1}_{Q\left(I\right)}\right\Vert_{L^2_c\left(\mathcal{H}%
;w\right)}^2=\frac{1}{2}\sum_{J\subset I}\left\vert b_J\right\vert^2=\frac{1%
}{2}\left\Vert \mathsf{Q}_I b\right\Vert_{L^2\left(\mathbb{R}\right)}^2.
\end{equation*}
We now provide an alternate, equivalent, way to study the backward testing
condition. Let $V:L^2\left(\mathbb{R}\right)\rightarrow L^2_{c}\left(%
\mathcal{H}\right)$ be the unitary operator defined in \eqref{e.changeofbase}
. 


Next, observe that $w\mathbf{1}_{Q\left (I\right)}=\sum_{J\subset I}
\left\vert b_{J}\right\vert^{2}\left\vert J\right\vert ^{-2}\mathbf{1}%
_{T(J)} $. Then, using these unitary operators we have that

\begin{eqnarray*}
\left\Vert \mathsf{Q}_I \mathsf{P}_{d}^{(0,0)}\mathsf{P}_{b}^{(1,0)} \mathsf{%
Q}_I \overline{b} \right\Vert_{L^2(\mathbb{R})} & = & \left\Vert \mathsf{Q}%
_I V^\ast \left(T^{(0,1,0,0)}_{b,d}\right)^{\ast} V\mathsf{Q}_I \overline{b}
\right\Vert_{L^2(\mathbb{R})} \\
& = & \left\Vert \mathsf{Q}_I V^\ast \left(T^{(0,1,0,0)}_{b,d}\right)^{\ast}
\left(\sum_{J\subset I} \overline{b_J}\, \widetilde{\mathbf{1}}%
_{T(J)}\right) \right\Vert_{L^2(\mathbb{R})} \\
& = & \left\Vert\mathsf{Q}_I V^\ast \mathcal{M}_{\overline{d}}^{\frac{1}{2}}%
\mathsf{U}^{\ast}\mathcal{M}_{b}^{-1} \left(\sum_{J\subset I} \overline{b_J}%
\, \widetilde{\mathbf{1}}_{T(J)}\right) \right\Vert_{L^2(\mathbb{R})} \\
& = & \left\Vert \mathsf{Q}_I V^\ast \mathcal{M}_{\overline{d}}^{\frac{1}{2}}%
\mathsf{U}^{\ast}\left( w\mathbf{1}_{Q(I)}\right) \right\Vert_{L^2(\mathbb{R}%
)} \\
& = & \left\Vert V^\ast \mathbf{1}_{Q(I)}\mathcal{M}_{\overline{d}}^{\frac{1%
}{2}}\mathsf{U}^{\ast}\left( w\mathbf{1}_{Q(I)}\right) \right\Vert_{L^2(%
\mathbb{R})} \\
& = & \left\Vert \mathbf{1}_{Q(I)}\mathcal{M}_{\overline{d}}^{\frac{1}{2}}%
\mathsf{U}^{\ast}\left( w\mathbf{1}_{Q(I)}\right) \right\Vert_{L^2_c(%
\mathcal{H})} \\
& = & \left\Vert \mathbf{1}_{Q(I)} \mathsf{U}^{\ast}\left( w\mathbf{1}%
_{Q(I)}\right) \right\Vert_{L^2_c(\mathcal{H};\sigma)}.
\end{eqnarray*}
These computation show that backward testing is equivalent to the following 
\begin{equation*}
\left\Vert \mathsf{Q}_I \mathsf{P}_{d}^{(0,0)}\mathsf{P}_{b}^{(1,0)} \mathsf{%
Q}_I \overline{b} \right\Vert_{L^2(\mathbb{R})}\lesssim \left\Vert \mathsf{Q}%
_I \overline{b}\right\Vert_{L^2\left(\mathbb{R}\right)}.
\end{equation*}
This condition is again clearly implied by the boundedness of the operator $%
\mathsf{P}_{b}^{\left( 0,1\right) }\circ \mathsf{P}_{d}^{\left( 0,0\right)}$
on $L^2\left(\mathbb{R}\right)$. Finally, we note that 
\begin{equation*}
\left\Vert \mathbf{1}_{Q(I)} \mathsf{U}^{\ast} (w\mathbf{1}%
_{Q(I)})\right\Vert_{L^2_c(\mathcal{H};\sigma)}^2=\sum_{J\subset I}\frac{%
\left\vert d_J\right\vert^2}{\left\vert J\right\vert}\left(\sum_{K\subset
J_+} \left\vert b_K\right\vert^2-\sum_{K\subset J_-} \left\vert
b_K\right\vert^2\right)^2.
\end{equation*}
To see this, note that 
\begin{eqnarray*}
\left\Vert \mathbf{1}_{Q(I)}\mathsf{U}^{\ast} (w\mathbf{1}%
_{Q(I)})\right\Vert_{L^2_c(\mathcal{H};\sigma)}^2 & = & 4\sum_{J\subset I} 
\frac{\left\vert d_J\right\vert^2}{\left\vert J\right\vert}
\left\vert\left\langle w \mathbf{1}_{Q(I)}, \mathbf{1}_{Q_{\pm}(J)}\right%
\rangle_{L^2(\mathcal{H})}\right\vert^2.
\end{eqnarray*}
But, observe that we have $\mathbf{1}_{Q_{\pm}(J)}=-\mathbf{1}_{Q(J_-)}+%
\mathbf{1}_{Q(J_+)}$, and that if $L\subset K$ 
\begin{equation*}
\mathbf{1}_{Q(L)}\mathbf{1}_{Q(K)}=\mathbf{1}_{Q(L)}.
\end{equation*}
Using these observations, we find that 
\begin{eqnarray*}
4\sum_{J\subset I} \frac{\left\vert d_J\right\vert^2}{\left\vert J\right\vert%
} \left\vert\left\langle w \mathbf{1}_{Q(I)}, \mathbf{1}_{Q_{\pm}(J)}\right%
\rangle_{L^2(\mathcal{H})}\right\vert^2 & = & 4\sum_{J\subset I} \frac{%
\left\vert d_J\right\vert^2}{\left\vert J\right\vert} \left\vert\left\langle
w, \mathbf{1}_{Q(J_+)}-\mathbf{1}_{Q(J_-)}\right\rangle_{L^2(\mathcal{H}%
)}\right\vert^2 \\
& = & \sum_{J\subset I}\frac{\left\vert d_J\right\vert^2}{\left\vert
J\right\vert}\left(\sum_{K\subset J_+} \left\vert
b_K\right\vert^2-\sum_{K\subset J_-} \left\vert b_K\right\vert^2\right)^2.
\end{eqnarray*}

Therefore, we have the following theorem providing the boundedness in terms
of testing conditions on the paraproduct $\mathsf{P}_{b}^{\left( 0,1\right)
}\circ \mathsf{P}_{d}^{\left( 0,0\right) }$. This is just a restatement of
Theorem \ref{Theorem_SIO}.

\begin{theorem}
The composition $\mathsf{P}_{b}^{\left( 0,1\right) }\circ \mathsf{P}%
_{d}^{\left( 0,0\right) }$ is bounded on $L^{2}\left( \mathbb{R}\right) $ if
and only if both%
\begin{eqnarray*}
\left\vert d_I\right\vert\left\Vert \mathsf{P}_{b}^{\left( 0,1\right) }
h_I\right\Vert_{L^2\left(\mathbb{R}\right)} & \leq & C_1; \\
\left\Vert \mathsf{Q}_I \mathsf{P}_{d}^{\left( 0,0\right) } \mathsf{P}%
_{b}^{\left( 1,0\right) }\mathsf{Q}_{I}\overline{b}\right\Vert _{L^{2}\left(%
\mathbb{R}\right)} & \leq & C_2 \left\Vert \mathsf{Q}_{I} b\right\Vert
_{L^{2}\left(\mathbb{R}\right)} \\
\end{eqnarray*}
for all $I\in\mathcal{D}$.; i.e. for all $I\in\mathcal{D}$ the following
inequalities are true 
\begin{eqnarray*}
\left\vert d_I\right\vert\left(\frac{1}{\left\vert I\right\vert}%
\sum_{L\subsetneq I} \left\vert b_L\right\vert^2\right)^{\frac{1}{2}} & \leq
& C_1; \\
\left(\sum_{J\subset I}\frac{\left\vert d_J\right\vert^2}{\left\vert
J\right\vert}\left(\sum_{K\subset J_+} \left\vert
b_K\right\vert^2-\sum_{K\subset J_-} \left\vert
b_K\right\vert^2\right)^2\right)^{\frac{1}{2}} & \leq & C_2
\left(\sum_{L\subset I}\left\vert b_{L}\right\vert^{2}\right)^{\frac{1}{2}}.
\end{eqnarray*}
Moreover, the norm of $\mathsf{P}_{b}^{\left( 0,1\right) }\circ \mathsf{P}%
_{d}^{\left( 1,0\right) }$ on $L^2\left(\mathbb{R}\right)$ satisfies 
\begin{equation*}
\left\Vert \mathsf{P}_{b}^{\left( 0,1\right) }\circ \mathsf{P}%
_{d}^{\left(0,0\right) }\right\Vert _{L^{2}\left( \mathbb{R}\right)
\rightarrow L^{2}\left( \mathbb{R}\right) }\approx C_1+C_2
\end{equation*}
where $C_1$ and $C_2$ are the best constants in appearing above.
\end{theorem}


\subsubsection{A Discrete T1 Theorem with Different Bases}

We will prove the following Theorem by adapting the proof strategy from
Nazarov, Treil and Volberg in \cite{NTV}. Recall that for $K\in\mathcal{D}$
that we have defined 
\begin{equation*}
\mathbf{1}_{Q_{\pm }\left( K\right) }\equiv -\sum_{L\subset K_{-}}\mathbf{1}%
_{T\left( L\right) }+\sum_{L\subset K_{+}}\mathbf{1}_{T\left( L\right) }.
\end{equation*}

\begin{theorem}
\label{NTV_Basis} Let 
\begin{equation*}
\mathsf{U}\equiv \sum_{K\in \mathcal{D}}\widetilde{\mathbf{1}}_{Q_{\pm
}\left( K\right) }\otimes \widetilde{\mathbf{1}}_{T\left( K\right) }
\end{equation*}
and suppose that $\mu $ and $\nu $ are positive measures on $\mathcal{H}$
that are constant on tiles, i.e., 
\begin{eqnarray*}
\mu & \equiv & \sum_{I\in\mathcal{D}}\mu_I\mathbf{1}_{T(I)} \\
\nu & \equiv & \sum_{I\in\mathcal{D}}\nu_I\mathbf{1}_{T(I)}.
\end{eqnarray*}
Then 
\begin{equation*}
\mathsf{U}\left( \mu \cdot \right):L^{2}_c\left(\mathcal{H}; \mu \right)
\rightarrow L^{2}_c\left(\mathcal{H};\nu \right)
\end{equation*}
if and only if both 
\begin{eqnarray*}
\left\Vert \mathsf{U}\left( \mu \mathbf{1}_{T\left( I\right) }\right)
\right\Vert _{L^{2}_c\left(\mathcal{H}; \nu \right) } &\leq & C_1\left\Vert 
\mathbf{1}_{T\left( I\right) }\right\Vert _{L^{2}_c\left(\mathcal{H}; \mu
\right) }=\sqrt{\mu \left( T\left( I\right) \right) }, \\
\left\Vert \mathbf{1}_{Q\left( I\right) } \mathsf{U}^{\ast }\left( \nu 
\mathbf{1}_{Q\left( I\right) }\right) \right\Vert _{L^{2}_c\left(\mathcal{H}%
; \mu \right) } &\leq & C_2\left\Vert \mathbf{1}_{Q\left( I\right)
}\right\Vert _{L^{2}_c\left(\mathcal{H}; \nu \right) }=\sqrt{\nu \left(
Q\left( I\right) \right) },
\end{eqnarray*}
hold for all $I\in \mathcal{D}$. Moreover, we have that 
\begin{equation*}
\left\Vert \mathsf{U}\right\Vert_{L^{2}_c\left(\mathcal{H}; \mu \right)
\rightarrow L^{2}_c\left(\mathcal{H};\nu \right)}\approx C_1+C_2
\end{equation*}
where $C_1$ and $C_2$ are the best constants appearing above.
\end{theorem}

\begin{proof}
Note that 
\begin{equation*}
\mathsf{U}_{\mu}(f)=\mathsf{U}\left( f\mu \right) =\sum_{K\in \mathcal{D}%
}\left\langle f\mu ,\widetilde{\mathbf{1}}_{T\left( K\right) }\right\rangle
_{L^2(\mathcal{H}) }\widetilde{\mathbf{1}}_{Q_{\pm }\left( K\right) }.
\end{equation*}%
For notational simplicity, in this proof only, we let $\nu(J)\equiv
\nu(Q(J)) $ (i.e., we implicitly identify $J$ with $Q(J)$). Now the weight
adapted orthonormal bases are given by%
\begin{equation*}
\left\{ h_{I}^{\mu }\right\} _{I\in \mathcal{D}}\text{ and }\left\{
H_{J}^{\nu }\right\} _{J\in \mathcal{D}},
\end{equation*}%
with 
\begin{equation*}
h_{I}^{\mu }\equiv \frac{\widetilde{\mathbf{1}}_{T\left( I\right) }}{\sqrt{%
\mu _{I}}}\text{ and }H_{J}^{\nu }\equiv \widetilde{\nu }\left( J\right)
\left( -\frac{\mathbf{1}_{Q\left( J_{+}\right) }}{\nu \left( J_{+}\right) }+%
\frac{\mathbf{1}_{Q\left( J_{-}\right) }}{\nu \left( J_{-}\right) }\right) ,
\end{equation*}%
where%
\begin{equation*}
\widetilde{\nu }\left( J\right) \equiv \sqrt{\frac{\nu \left( J_{+}\right)
\nu \left( J_{-}\right) }{\nu \left( I_{+}\right) +\nu \left( I_{-}\right) }}%
.
\end{equation*}

Let $\widehat{f}_\mu$ denote the ``Haar coefficient'' of $f$ with respect to
the basis $h_I^\mu$, i.e., 
\begin{equation*}
\widehat{f}_\mu(I)\equiv\left\langle f,h_I^{\mu}\right\rangle_{L^2(\mathcal{H%
};\mu)},
\end{equation*}
and similarly for $\widehat{g}_\nu(J)$. We can now expand the function $f$
and $g$ with respect to these weighted orthonormal bases and write $%
f=\sum_{I\in\mathcal{D}}\widehat{f}_\mu\left( I\right) h_{I}^{\mu }$ and $%
g=\sum_{J\in\mathcal{D}}\widehat{g}_\nu\left( J\right) H_{\nu }^{J}$. Doing
so, we then see that 
\begin{eqnarray*}
\left\langle \mathsf{U}_{\mu }f,g\right\rangle _{L^2_c\left(\mathcal{H}%
;\nu\right) } &=&\sum_{I,J\in\mathcal{D}}\widehat{f}_\mu\left( I\right) 
\widehat{g}_{\nu}\left( J\right) \left\langle \mathsf{U}_{\mu }h_{I}^{\mu
},H_{\nu }^{J}\right\rangle _{L^2_c(\mathcal{H};\nu) } \\
&=&\sum_{I,J\in\mathcal{D}}\widehat{f}_{\mu}\left( I\right) \widehat{g}%
_{\nu}\left( J\right) \sqrt{\mu _{I}}\left\langle \widetilde{\mathbf{1}}%
_{Q_{\pm }\left( I\right) },H_{\nu }^{J}\right\rangle _{L^2_c(\mathcal{H}%
;\nu) }
\end{eqnarray*}%
since $\mathsf{U}_{\mu }h_{I}^{\mu }=\sqrt{\mu _{I}}\,\widetilde{\mathbf{1}}%
_{Q_{\pm }\left( I\right) }$. By a further, straightforward, computation we
have%
\begin{eqnarray}
\left\langle \widetilde{\mathbf{1}}_{Q_{\pm }\left( I\right) },H_{\nu
}^{J}\right\rangle _{L^2_c(\mathcal{H};\nu) } &=&\frac{1}{\left\vert
I\right\vert }\int \left( -\mathbf{1}_{Q\left( I_{+}\right) }+\mathbf{1}%
_{Q\left( I_{-}\right) }\right) \widetilde{\nu }\left( J\right) \left( -%
\frac{\mathbf{1}_{Q\left( J_{+}\right) }}{\nu \left( J_{+}\right) }+\frac{%
\mathbf{1}_{Q\left( J_{-}\right) }}{\nu \left( J_{-}\right) }\right) \nu dA 
\notag \\
&=&\left\{ 
\begin{array}{ccc}
0 & \text{ if } & J\subset I, \\ 
\frac{\pm 1}{\left\vert I\right\vert }\frac{\widetilde{\nu }\left( J\right) 
}{\nu \left( J_{\pm }\right) }\left( -\nu \left( I_{+}\right) +\nu \left(
I_{-}\right) \right) & \text{ if } & I\subset J_{\pm }, \\ 
\frac{1}{\left\vert I\right\vert }\widetilde{\nu }\left( I\right) & \text{
if } & I=J.%
\end{array}
\right.  \label{MatrixComp}
\end{eqnarray}%
Altogether we have 
\begin{eqnarray*}
\left\langle \mathsf{U}_{\mu }f,g\right\rangle _{L^2_c(\mathcal{H};\nu) }
&=&\sum_{I,J\in\mathcal{D}}\widehat{f}_\mu\left( I\right) \widehat{g}%
_\nu\left( J\right) \sqrt{\mu _{I}}\left\langle \widetilde{\mathbf{1}}%
_{Q_{\pm }\left( I\right) },H_{\nu }^{J}\right\rangle _{L^2_c(\mathcal{H}%
;\nu) } \\
&=&\left(\sum_{I=J}+\sum_{J\subset I}+\sum_{I\subset J}\right)\widehat{f}%
_\mu\left( I\right) \widehat{g}_\nu\left( J\right) \sqrt{\mu _{I}}%
\left\langle \widetilde{\mathbf{1}}_{Q_{\pm }\left( I\right) },H_{\nu
}^{J}\right\rangle _{L^2_c(\mathcal{H};\nu) } \\
& \equiv & \mathbf{A}+\mathbf{B}+\mathbf{C}.
\end{eqnarray*}%
We then need to show that 
\begin{equation*}
\left\vert \left\langle \mathsf{U}_{\mu }f,g\right\rangle _{L^2_c(\mathcal{H}%
;\nu) }\right\vert \lesssim \left(C_1+C_2\right) \left\Vert f\right\Vert
_{L^{2}_c\left( \mathcal{H};\mu \right) }\left\Vert g\right\Vert
_{L^{2}_c\left( \mathcal{H}; \nu \right) }
\end{equation*}
and to accomplish this we will show the desired estimates on each of $%
\mathbf{A}$, $\mathbf{B}$ and $\mathbf{C}$.

Now for the first term, by the third line in \eqref{MatrixComp} we have that 
\begin{eqnarray*}
\left\vert \mathbf{A}\right\vert &=& \left\vert \sum_{I\in\mathcal{D}}%
\widehat{f}_\mu\left( I\right) \widehat{g}_\nu\left( I\right) \frac{\sqrt{%
\mu _{I}}}{\left\vert I\right\vert }\widetilde{\nu }\left(
I\right)\right\vert \\
&\leq & \left\Vert f\right\Vert _{L^{2}_c\left( \mathcal{H};\mu \right)
}\left\Vert g\right\Vert _{L^{2}_c\left( \mathcal{H}; \nu \right) }\left(
\sup_{I\in\mathcal{D}}\frac{\sqrt{\mu _{I}}}{\left\vert I\right\vert }%
\widetilde{\nu }\left( I\right) \right),
\end{eqnarray*}%
with the last line following by Cauchy-Schwarz and Parseval's Identity.
However, the forward testing condition gives%
\begin{eqnarray*}
\frac{C_1^2}{2}\mu _{I}\left\vert I\right\vert^2 & = & C_1^2\left\Vert 
\mathbf{1}_{T\left( I\right)} \right\Vert _{L^{2}\left(\mathcal{H}%
;\mu\right)}^{2} \\
& \geq & \left\Vert \mathsf{U}_{\mu }\left( \mathbf{1}_{T\left( I\right)
}\right) \right\Vert _{L^{2}\left(\mathcal{H};\nu\right)}^{2}=8 \left\Vert
\mu_{I} \mathbf{1}_{Q_{\pm }\left( I\right) }\right\Vert_{L^2_c\left(%
\mathcal{H};\nu\right)}^2 \\
& = & 8 \mu_{I}^{2}\left( \nu \left( I_{+}\right) +\nu \left( I_{-}\right)
\right) ,
\end{eqnarray*}
Then, using 
\begin{equation*}
\widetilde{\nu }\left( I\right) ^{2}\equiv \frac{\nu \left( I_{+}\right) \nu
\left( I_{-}\right) }{\nu \left( I_{+}\right) +\nu \left( I_{-}\right) }\leq
\min \left\{ \nu \left( I_{+}\right) ,\nu \left( I_{-}\right) \right\} \leq
\nu \left( I_{+}\right) +\nu \left( I_{-}\right) ,
\end{equation*}%
we get%
\begin{equation*}
\sup_{I\in\mathcal{D}}\frac{\sqrt{\mu _{I}}}{\left\vert I\right\vert }%
\widetilde{\nu }\left( I\right) \lesssim C_1,
\end{equation*}
and thus, have 
\begin{equation*}
\left\vert \mathbf{A}\right\vert \lesssim C_1\left\Vert f\right\Vert
_{L^{2}_c\left( \mathcal{H};\mu \right) }\left\Vert g\right\Vert
_{L^{2}_c\left( \mathcal{H}; \nu \right) }.
\end{equation*}

The second term is trivial since $\mathbf{B}=0$ by the first line in %
\eqref{MatrixComp}. Finally, by the second line in \eqref{MatrixComp} we
have 
\begin{eqnarray*}
\left\vert \mathbf{C}\right\vert &=& \left\vert \sum_{I\in\mathcal{D}%
}\sum_{J\succ I}\widehat{f}_\mu\left( I\right) \widehat{g}_\nu\left(
J\right) \frac{\sqrt{\mu _{I}}}{\left\vert I\right\vert }\left( \frac{\pm 
\widetilde{\nu }\left( J\right) }{\nu \left( J_{\pm }\right) }\right) \left(
-\nu \left( I_{+}\right) +\nu \left( I_{-}\right) \right) \right\vert \\
&=& \left\vert \sum_{I\in\mathcal{D}}\widehat{f}_{\mu}\left( I\right) \frac{%
\sqrt{\mu _{I}}}{\left\vert I\right\vert }\left( -\nu \left( I_{+}\right)
+\nu \left( I_{-}\right) \right) \left( \sum_{J\succ I}\widehat{g}%
_{\nu}\left( J\right) \frac{\pm \widetilde{\nu }\left( J\right) }{\nu \left(
J_{\pm }\right) }\right) \right\vert \\
&=&\left\vert \sum_{I\in\mathcal{D}}\widehat{f}_\mu\left( I\right) \frac{%
\sqrt{\mu _{I}}}{\left\vert I\right\vert }\left( -\nu \left( I_{+}\right)
+\nu \left( I_{-}\right) \right) \left\langle g,\frac{\mathbf{1}_{Q\left(
I\right) }}{\nu \left( I\right) }\right\rangle _{L^2(\mathcal{H}; \nu)
}\right\vert \\
&\leq &\left( \sum_{I\in\mathcal{D}}\left\vert \widehat{f}_{\mu}\left(
I\right) \right\vert ^{2}\right) ^{\frac{1}{2}} \left( \sum_{I\in\mathcal{D}%
}\left\vert \left\langle g,\frac{\mathbf{1}_{Q\left( I\right) }}{\nu \left(
I\right) }\right\rangle _{L^2(\mathcal{H};\nu) }\right\vert ^{2}\frac{\mu
_{I}}{\left\vert I\right\vert ^{2}}\left( -\nu \left( I_{+}\right) +\nu
\left( I_{-}\right) \right) ^{2}\right) ^{\frac{1}{2}} \\
& = & \left\Vert f\right\Vert_{L^2_c(\mathcal{H})} \left( \sum_{I\in\mathcal{%
D}}\left\vert \left\langle g,\frac{\mathbf{1}_{Q\left( I\right) }}{\nu
\left( I\right) }\right\rangle _{L^2(\mathcal{H};\nu) }\right\vert ^{2}\frac{%
\mu _{I}}{\left\vert I\right\vert ^{2}}\left( -\nu \left( I_{+}\right) +\nu
\left( I_{-}\right) \right) ^{2}\right) ^{\frac{1}{2}}.
\end{eqnarray*}
Expanding $\mathsf{U}_\nu^{\ast}\left(\mathbf{1}_{Q(I)}\right)$ with respect
to the basis $\left\{\widetilde{\mathbf{1}}_{T(J)}\right\}_{J\in\mathcal{D}}$
we note that the backward testing condition gives%
\begin{eqnarray*}
C_2^2\nu \left( I\right) & = & C_2^2\left\Vert \mathbf{1}_{Q(I)}\right%
\Vert_{L^2_c(\mathcal{H};\nu)}^2 \\
& \geq & \left\Vert \mathbf{1}_{Q\left( I\right) } \mathsf{U}_{\nu }^{\ast
}\left( \mathbf{1}_{Q\left( I\right) }\right) \right\Vert _{L^{2}_c\left( 
\mathcal{H};\mu \right) }^{2} \\
& = & \sum_{J\subset I} \mu_J \left\vert \left\langle \nu \mathbf{1}_{Q(I)}, 
\widetilde{\mathbf{1}}_{Q_{\pm}(J)}\right\rangle_{L^2(\mathcal{H}%
)}\right\vert^2 \\
& = & \sum_{J\subset I}\frac{\mu _{J}}{\left\vert J\right\vert ^{2}}\left(
-\nu \left( J_{+}\right) +\nu \left( J_{-}\right) \right) ^{2},
\end{eqnarray*}
and then the Carleson Embedding Theorem shows that%
\begin{equation*}
\sum_{I\in\mathcal{D}}\left\vert \left\langle g,\frac{\mathbf{1}_{Q\left(
I\right) }}{\nu \left( I\right) }\right\rangle _{L^2_c(\mathcal{H};\nu)
}\right\vert ^{2}\frac{\mu _{I}}{\left\vert I\right\vert ^{4}}\left( -\nu
\left( I_{+}\right) +\nu \left( I_{-}\right) \right) ^{2}\lesssim
C_2^2\left\Vert g\right\Vert _{L^{2}_c\left( \mathcal{H};\nu \right) }^{2}.
\end{equation*}
Therefore, we have 
\begin{equation*}
\left\vert \mathbf{C}\right\vert\lesssim C_2\left\Vert f\right\Vert_{L^2_c(%
\mathcal{H})}\left\Vert g\right\Vert_{L^2_c(\mathcal{H})}
\end{equation*}
Combining the above we get%
\begin{equation*}
\left\vert \left\langle \mathsf{U}_{\mu }f,g\right\rangle _{L^2_c(\mathcal{H}%
;\nu) }\right\vert \leq \left\vert \mathbf{A}\right\vert +\left\vert \mathbf{%
C}\right\vert \lesssim \left(C_1+C_2\right) \left\Vert f\right\Vert
_{L^{2}\left( \mathcal{H}; \mu \right) }\left\Vert g\right\Vert
_{L^{2}\left( \mathcal{H};\nu \right) }.
\end{equation*}
\end{proof}

\begin{remark}
The paper \cite{NTV} more generally studies operators that are ``well
localized'' with respect to the Haar basis. It is clear that the method of
proof in \cite{NTV} can be extended to operators that are sufficiently
localized with respect to a pair of bases. We do not explore this extension
at this time, but will return to it at some point in the future.
\end{remark}

\section{Conclusion}

Unfortunately the methods we have used in this paper do not appear to work
to handle type $(0,1,0,1)$ compositions. However, we strongly believe that
the following conjecture is true:

\begin{conjecture}
$\mathsf{P}_{b}^{\left( 0,1\right) }\circ \mathsf{P}_{d}^{\left( 0,1\right)
} $ is bounded on $L^{2}\left( \mathbb{R}\right) $ if and only if for each $%
I\in\mathcal{D}$ there exists $L^2\left(\mathbb{R}\right)$ functions $F_I$
and $B_I$ of norm $1$ such that 
\begin{eqnarray*}
\left\Vert \mathsf{P}_{b}^{\left( 0,1\right) }\circ \mathsf{P}_{d}^{\left(
0,1\right)} F_I \right\Vert_{L^{2}\left( \mathbb{R}\right)} & \leq & C_1 \\
\left\Vert \mathsf{P}_{d}^{\left( 1,0\right) }\circ \mathsf{P}_{b}^{\left(
1,0\right)} B_I \right\Vert_{L^{2}\left( \mathbb{R}\right)} & \leq & C_2.
\end{eqnarray*}
Moreover, we will have 
\begin{equation*}
\left\Vert \mathsf{P}_{b}^{\left( 0,1\right) }\circ \mathsf{P}_{d}^{\left(
0,1\right)} \right\Vert_{L^{2}\left( \mathbb{R}\right)\rightarrow L^{2}\left(%
\mathbb{R}\right)}\approx C_1+C_2.
\end{equation*}
\end{conjecture}

The choice of the families $\{F_I\}_{I\in\mathcal{D}}$ and $\{B_I\}_{I\in%
\mathcal{D}}$ will clearly play an important role.


\begin{bibdiv}
\begin{biblist}


\bib{ArRoSa}{article}{
   author={Arcozzi, Nicola},
   author={Rochberg, Richard},
   author={Sawyer, Eric},
   title={Carleson measures for analytic Besov spaces},
   journal={Rev. Mat. Iberoamericana},
   volume={18},
   date={2002},
   number={2},
   pages={443--510}
}

\bib{Naz}{article}{
   author={Nazarov, F.},
   title={A counterexample to Sarason's conjecture},
   eprint={http://www.math.msu.edu/~fedja/Preprints/Sarason.ps},
   status={preprint},
   pages={1--17},
   date={1997}
}

\bib{NTV}{article}{
   author={Nazarov, F.},
   author={Treil, S.},
   author={Volberg, A.},
   title={Two weight inequalities for individual Haar multipliers and other
   well localized operators},
   journal={Math. Res. Lett.},
   volume={15},
   date={2008},
   number={3},
   pages={583--597}
}

\bib{Petermichl}{article}{
   author={Petermichl, S.},
   title={The sharp bound for the Hilbert transform on weighted Lebesgue
   spaces in terms of the classical $A_p$ characteristic},
   journal={Amer. J. Math.},
   volume={129},
   date={2007},
   number={5},
   pages={1355--1375}
}

\bib{Petermichl2}{article}{
   author={Petermichl, Stefanie},
   title={Dyadic shifts and a logarithmic estimate for Hankel operators with
   matrix symbol},
   language={English, with English and French summaries},
   journal={C. R. Acad. Sci. Paris S\'er. I Math.},
   volume={330},
   date={2000},
   number={6},
   pages={455--460}
}

\bib{PSRW}{article}{
   author={Pott, Sandra},
   author={Sawyer, Eric T.},
   author={Reguera, Maria C.},
   author={Wick, Brett D.},
   title={The Linear Bound for the Hilbert Transform},
   status={preprint}
}


\bib{PS}{article}{
   author={Pott, Sandra},
   author={Smith, Martin P.},
   title={Paraproducts and Hankel operators of Schatten class via
   $p$-John-Nirenberg theorem},
   journal={J. Funct. Anal.},
   volume={217},
   date={2004},
   number={1},
   pages={38--78}
}

\bib{sarasonConj}{article}{
  author={Sarason, Donald},
  title={Products of Toeplitz operators},
  book={title={Linear and complex analysis. Problem book 3. Part I}, series={Lecture Notes in Mathematics}, volume={1573}, editor={Havin, V. P.}, editor={Nikolski, N. K.}, publisher={Springer-Verlag}, place={Berlin}, date={1994}, },
  pages={318-319},
}

\bib{Sawyer}{article}{
   author={Sawyer, Eric T.},
   title={A characterization of two weight norm inequalities for fractional
   and Poisson integrals},
   journal={Trans. Amer. Math. Soc.},
   volume={308},
   date={1988},
   number={2},
   pages={533--545}
}

\bib{TVZ}{article}{
   author={Treil, Sergei},
   author={Volberg, Alexander},
   author={Zheng, Dechao},
   title={Hilbert transform, Toeplitz operators and Hankel operators, and
   invariant $A_\infty$ weights},
   journal={Rev. Mat. Iberoamericana},
   volume={13},
   date={1997},
   number={2},
   pages={319--360}
}

\end{biblist}
\end{bibdiv}
\end{document}